\documentclass[11pt]{article}
\usepackage{amsmath,amssymb,amsthm,amsfonts,hhline,color}

\def\H{{\mathbb H}}
\def\C{{\mathbb C}}

\def\Z{{\mathbb Z}}

\def\Q{{\mathbb Q}}

\def\N#1{{\cal N}\!=\!#1}

\def\qed{\hfill\framebox[2.5mm][t1]{\phantom{x}}}

\title{Mathieu Moonshine and the Geometry of K3 Surfaces}

\author{Thomas Creutzig\thanks{Department of Mathematical and Statistical Sciences, University of Alberta,
Edmonton, Alberta  T6G 2G1, Canada. 
email: creutzig@ualberta.ca}\ \,
and\ \, Gerald H\"ohn\thanks{Kansas State University, 138 Cardwell Hall, Manhattan, KS 66506-2602, USA. 
email: \hbox{gerald@math.ksu.edu}}} 

\date{}

\begin{document}

\bibliographystyle{amsalpha}

\theoremstyle{plain}
\newtheorem*{introthm}{Theorem}
\newtheorem{thm}{Theorem}[section]
\newtheorem{prop}[thm]{Proposition}
\newtheorem{lem}[thm]{Lemma}
\newtheorem{cor}[thm]{Corollary}
\newtheorem{conj}[thm]{Conjecture}

\theoremstyle{definition}
\newtheorem{defi}[thm]{Definition}
\newtheorem{rem}[thm]{Remark}

\newcommand {\tr}{\text{tr}}

\renewcommand{\baselinestretch}{1.2}

\maketitle

\begin{abstract}
We compare the moonshine observation of Eguchi, Ooguri and Tachikawa 
relating the Mathieu group $M_{24}$ and the complex elliptic genus of a K3 surface
with the symmetries of geometric structures on K3 surfaces. 

Two main results are that the complex elliptic genus of a K3 surface is a virtual module for the Mathieu
group $M_{24}$ and also for a certain vertex operator superalgebra $V^G$ where $G$ is the holonomy group. 
\end{abstract}

\section{Introduction}


Moonshine refers to an unexpected relation between finite simple groups, automorphic forms and physics, 
with the most prominent example of monstrous moonshine~\cite{CN,B}. 
A few years ago, a similar observation was made concerning the Mathieu group $M_{24}$
and the complex elliptic genus of a K3 surface. 

\subsection{Mathieu Moonshine}

Eguchi, Ooguri and Tachikawa \cite{EOT} observed that 
the sum decomposition of the Jacobi form $2\,\phi_{0,1}(z;\tau)$
of weight~$0$ and index~$1$ into the (expected) characters
of the $\N4$ super Virasoro algebra at central charge $c=6$ has multiplicities which are 
simple sums of dimensions of irreducible representations of the largest Mathieu group $M_{24}$.
The Jacobi form $2\,\phi_{0,1}(z;\tau)$ has several interpretations, it is the complex elliptic genus $\chi_{-y}(q,{\cal L}X)$
of a K3 surface $X$ and in string theory it has the meaning of the partition function of physical states in the Ramond-sector 
of a  sigma model on $X$. 

The observation of Eguchi, Ooguri and Tachikawa suggests the existence of a graded
$M_{24}$-module $K=\bigoplus_{n=0}^{\infty} K_n\, q^{n-1/8}$ whose graded character is given by the sum decomposition of 
the elliptic genus into characters of the $\N4$ super Virasoro algebra.   
Subsequently the analogues to McKay-Thompson series in monstrous moonshine where proposed in
several works~\cite{C, CD, EH, GHV, GHV2}. 
The McKay-Thompson series for $g$ in $M_{24}$ are of the form
\begin{equation}\label{thompsonseries}
\Sigma_g(q)\ =\ q^{-1/8}\,\sum_{n=0}^{\infty} {\rm Tr}(g|K_n)\, q^{n} \ = \
\frac{e(g)}{24}\, \Sigma(q) - \frac{f_g(q)}{\eta(q)^3}
\end{equation}
where  $\Sigma=\Sigma_e$ is the graded dimension of $K$ (an explicit mock modular form), $e(g)$ is the character of
the $24$-dimensional permutation representation of $M_{24}$, the series $f_g$ is a certain explicit modular form of weight~$2$ for some
subgroup $\Gamma_0(N_g)$ of ${\rm SL}(2,\Z)$ and $\eta$ is the Dedekind eta function.
Terry Gannon~\cite{Gannon} has proven that these McKay-Thompson series 
indeed determine a $M_{24}$-module:
\begin{thm}\label{thm:gannon}
The McKay-Thompson series as in~\cite{GHV,EH} determine a virtual graded
$M_{24}$-module $K=\bigoplus_{n=0}^{\infty} K_n\, q^{n-1/8}$. For 
$n\geq 1$ the $K_n$ are honest (and not only virtual) $M_{24}$-representations.  
\end{thm}
The proof uses techniques of modular forms and group theory. It does not explain any connection to geometry or physics. 
Since the complex elliptic genus is both an invariant of K3 surfaces and a string theory object, 
it is believed that there
is a geometric and/or physics explanation of the Mathieu moonshine observation. 
In this work, we investigate implications of the geometry to the moonshine.  

\subsection{The complex elliptic genus}

The complex elliptic genus appears originally in the physical literature. 
A rigorous mathematical definition and a first investigation of its basic topological properties, like rigidity
have been first given by Krichever~\cite{Krichever} and the second author~\cite{Ho-Diplom}
following similar work by Ochanine, Witten, Hirzebruch and Bott/Taubes for other elliptic
genera.
The complex elliptic genus equals the graded dimension of the cohomology 
of the chiral de Rham complex of a complex manifold~\cite{MSV,BoLi}.

We recall that a complex genus in the sense of Hirzebruch~\cite{Hirzebruch-Habil} is a graded ring homomorphism
from the complex bordism ring into some other graded ring $R$. 
The complex elliptic genus can be formally defined as the $S^1$-equivariant $\chi_y$-genus 
of the loop space of a manifold
$$\chi_y(q,{\cal L}X):= (-y)^{-d/2}\chi_y\bigl(X,\bigotimes_{n=1}^{\infty}\Lambda_{yq^n}T^*
\otimes \bigotimes_{n=1}^{\infty}\Lambda_{y^{-1}q^n}T
\otimes \bigotimes_{n=1}^{\infty}S_{q^n}(T^*\oplus T)\bigr) $$
having values in $\Q[y^{1/2},y^{-1/2}][[q]]$.
For manifolds $X$ with vanishing first Chern class it is the Fourier expansion of a 
Jacobi form of weight~$0$ and index half the complex dimension of $X$ (see~\cite{Ho-Diplom}). 
For automorphisms $g$ of $X$ one has the corresponding equivariant elliptic genus $\chi_y(g;q,{\cal L}X)$.


The Jacobi form $2\,\phi_{0,1}(z;\tau)$ equals the complex elliptic genus of K3 surfaces 
(cf.\ for example~\cite{Ho-Diplom}, p.~62).


\subsection{Results and Organization}

By combining Mukai's~\cite{Mukai} classification of groups acting symplectically on K3 surfaces 
and information from rational character tables we prove in Section~\ref{Mukaigroups} 
(see Corollary~\ref{cor:virtualM24ellipticgenus}):
\begin{introthm}
The complex elliptic genus of a K3 surface can be given the structure of a
virtual $M_{24}$-module which is compatible with the $H$-module structure
for all possible groups $H$ of symplectic automorphisms of K3 surfaces under restriction.
\end{introthm}
We remark that all such groups $H$ of symplectic automorphisms of K3 surfaces are subgroups
of $M_{24}$.

\smallskip

We then explain in Section~\ref{Holonomy} how to construct bundles over a K3 surface that at each fiber have
the structure of a (possibly twisted) module $M$ for a vertex operator superalgebra $V$ and that the corresponding
twisted Todd genus is a character of a virtual module for a vertex operator subalgebra $V^G$ 
where $G$ is the structure group of the bundle (see Theorem~\ref{character}).
For the case of the complex elliptic genus of a K3 surface we have (Corollary~\ref{cor:voa}):
\begin{introthm}
The complex elliptic genus of a K3 surface is the graded character of a natural
virtual vertex operator algebra module for $V^G$ and its $\N4$ super Virasoro vertex subalgebra.
\end{introthm}
We also obtain explicit formulas for the decomposition of $M$ into characters of the $\N4$ super Virasoro algebra.

\smallskip

The equivariant complex elliptic genus of a K3 surface can be computed by using the holomorphic 
Lefschetz fixed-point formula. On the other hand, the known Mathieu moonshine  uses the McKay-Thompson
series, i.e.\ equivariant elliptic genera in the sense of string theory and super conformal field theory 
on K3 surfaces. 
A priori these two could be different, but fortunately we find (Theorem~\ref{moonshinegeometric}):
\begin{introthm}
For a finite symplectic automorphism $g$ acting on a K3 surface $X$
the equivariant elliptic genus and the $\N4$ character determined by the McKay-Thompson series of Mathieu moonshine
agree.
\end{introthm}
We also compare in this final section the geometric structures for K3 surfaces which we have considered 
with the proposed Mathieu moonshine. 

\bigskip

\noindent {\bf Acknowledgements} We thank the referees for the very useful comments. 

%

\section{Symplectic group actions on K3 surfaces and the Mathieu group $M_{24}$}\label{Mukaigroups}

A K3 surface is a simply connected complex surface with trivial canonical line bundle. 
An action of a group $H$  on $X$ by complex automorphisms is called symplectic
if the induced action on the space $H^0(X,\Omega^2)$ of holomorphic $2$-forms is trivial. 
For a symplectic action 
one can find a hyperk\"ahler metric $g$ on $X$ for which $H$ acts by isometries and which
is a K\"ahler metric for the given complex structure. Recall that a metric for a four dimensional
manifold is called hyperk\"ahler if the holonomy group is contained in 
${\rm Sp}(1)\cong {\rm SU}(2)$. Conversely, given a hyperk\"ahler metric $g$ on $X$ there is a family
of complex structures on the underlying differentiable manifold parameterized by a $2$-sphere and
the action of a group $H$ by isometries is symplectic for all of them.

Let $g$ be a non-trivial symplectic automorphism of finite order acting on a K3 surface $X$. 
By~\cite{Nikulin} such automorphisms have only isolated fixed points. The holomorphic Lefschetz
fixed point formula~\cite{AtiyahBott} shows that only seven cases are possible~\cite{Nikulin} 
(see also~\cite{Mukai}). We list them together with the eigenvalues of $g$ acting on the
tangent spaces in Table~\ref{finitesymp}. Here, $\zeta_n$ stands for a primitive $n$-th
root of unity.

\begin{table}\caption{Fixed point sets of non-trivial symplectic automorphisms $g$}\label{finitesymp}
\centerline{\begin{tabular}{cccc}
order of $g$  & \# of fixed points & eigenvalues of $g$ at the fixed points \\ \hline
$2$   & 8 & $8 \times (-1,-1)$\\
$3$     & 6 & $6 \times (\zeta_3,\zeta_3^{-1})$\\
$4$     & 4 & $4 \times (i,-i)$ \\
$5$     & 4 & $2 \times (\zeta_5,\zeta_5^{-1})$ and $ 2 \times (\zeta_5^2,\zeta_5^{-2})$\\
$6$     & 2 & $2 \times (\zeta_6,\zeta_6^{-1})$\\
$7$    & 3 & $(\zeta_7,\zeta_7^{-1})$, $(\zeta_7^2,\zeta_7^{-2})$ and  $(\zeta_7^3,\zeta_7^{-3})$\\
$8$      & 2 & $(\zeta_8,\zeta_8^{-1})$ and $(\zeta_8^3,\zeta_8^{-3})$\\
\end{tabular}}
\end{table}

\medskip

The finite symplectic automorphism groups of K3 surfaces have been classified by Mukai.
They are all isomorphic to subgroups of the Mathieu group $M_{23}$ of a particular type.
We recall that $M_{23}$ is isomorphic to a one-point stabilizer for the permutation action 
of $M_{24}$  on $24$ elements.

\begin{thm}[Mukai\ \cite{Mukai}]
A finite group $H$ acting symplectically on a K3 surface is isomorphic to
a subgroup of $M_{23}$ with at least five orbits on the regular
permutation representation of the Mathieu group $M_{24}\supset M_{23}$ on $24$ elements.
\end{thm}

\begin{rem}\label{mukaigroups}
\begin{enumerate}
\item Every such group $H$ is a subgroup of $11$ maximal groups explicitly given in\ \cite{Mukai}.
They are of the following type (cf.~\cite{Mukai} for details):
{\setlength{\arraycolsep}{4pt} 
$$\hspace{-1cm}\begin{array}{c|cccccccccccc}
\hbox{no.}       & 1 & 2 & 3 & 4 & 5 & 6 & 7 & 8 & 9 & 10 & 11 \\ \hline
\hbox{structure} & L_2(7) & A_6 & S_5 & 2^4A_5 & 2^4S_4 & A_{4,4} & T_{192} & 2^4D_{12} & 3^2D_8 & 3^2Q_8 & SL_2(3)\\
\hbox{order} & 168 & 360 & 120 & 960 & 384 & 288 & 192 & 192 & 72  & 72 & 48
\end{array} $$
}

\item The non-trivial minimal subgroups correspond to the seven possibilities of
finite symplectic groups actions as listed in Table~\ref{finitesymp}. Below, 
we list the ATLAS names~\cite{atlas} of the corresponding generators (as well as the trivial
element) as elements in the two
Mathieu groups $M_{23}$, $M_{24}$ and the Conway group ${\rm Co}_0$ which contains
$M_{24}$ as a subgroup. 
$$\begin{array}{c|cccccccc}
\hbox{order} &  1 & 2 & 3 & 4 & 5 & 6 & 7 & 8 \\ \hline
M_{23} & 1A & 2A & 3A & 4A & 5A & 6A & 7A,\, 7B & 8A \\
M_{24} &  1A & 2A & 3A  & 4B & 5A & 6A & 7A,\, 7B & 8A  \\
{\rm Co}_0 & 1A+ & 2A+ & 3B+ & 4C+ & 5B+ & 6E+ & 7B+ & 8E+  \\
\end{array}$$

\item 
There are in total $80$ abstract isomorphism types of symplectic automorphism groups $H$ 
(including the trivial one) acting on K3 surfaces.
Considered as a subgroup of $O(L)$, where
$L=H^2(X,\Z)$ is the intersection lattice of $X$, the group is unique up to conjugacy
in $85$ of those cases.
For five cases, there exist two possibilities, each; see~\cite{Hashimoto}.
\end{enumerate}
\end{rem}

Given a holomorphic vector bundle $E$ on a K3 surface $X$ associated to the tangent bundle,
and a finite automorphism group $H$ of $X$, the induced action of $H$ on the 
cohomology groups $H^q(X,E)$ makes the virtual vector space $\chi(X, E)=\bigoplus_{q=0}^2(-1)^q H^q(X,E)$
into a virtual $H$-module. For $H$ symplectic, we know from the holomorphic
Lefschetz formula and Table~\ref{finitesymp} that the value of 
$$\chi(g;X,E)=\tr (g|H^0(X,E)) - \tr (g|H^1(X,E)) + \tr (g|H^2(X,E))$$
depends only on the order of $g\in H$ and is rational. In particular, we see that the structure of the virtual
$H$-module $\chi(X,E)$ for all eleven maximal Mukai groups $H$ is completely determined
by the value of $\chi(g;X,E)$ for elements $g$ of order $1$, $2$, \ldots, $8$.
We axiomatize this situation in the following definition.

\begin{defi}
A {\it Mukai\/} (resp.\ {\it virtual Mukai\/}) {\it representation\/} $\rho$ is a family 
$(\rho_i)_{i=1,\ldots,11}$ of rational characters (resp.\ generalized rational characters)
of the eleven maximal symplectic groups $H_i$, $i=1$, $\ldots$, $11$, such that
$\rho_i(g)$ for $g\in H_i$ depends only on the order of $g$ and this value
is independent of $i$.
\end{defi}

\begin{thm}\label{virtualM24}
Let $\rho=(\rho_i)_{i=1,\ldots,11}$ be a virtual Mukai representation. 
Then there exists a virtual rational representation $\mu$ of $M_{24}$ 
such that $\mu|_{H_i}=\rho_i$ for each of the eleven Mukai
groups $H_i$.
\end{thm}

\begin{rem}
\begin{enumerate}
\item The embedding of $H_i$ into $M_{24}$ may not be unique 
but the value of the characters is, i.e.\ $\mu|_{H_i}$ is always
well-defined.
\item The representation $\mu$ of $M_{24}$ is not unique. In fact,
it can be changed by any virtual rational representation of $M_{24}$ with the 
property that the character for the eight conjugacy classes of elements belonging
to symplectic automorphism is zero.
\end{enumerate}
\end{rem}

\noindent{\bf Proof:} 
Let $L_i$ be the $\Z$-module generated by the rational characters of the group $H_i$. We consider
the submodule $L_i^0$ of $L_i$ defined by the property that the values for those conjugacy classes of $H_i$
for which the elements have the same order are equal. 
Let $M$ be the $\Z$-module of $\Z$-valued functions on the set $\{1,\,2,\,\ldots,\, 8\}$.
We let $N_i$ be the submodule of $M$ consisting
of those functions on $\{1,\,2,\,\ldots,\, 8\}$ for which the restriction to the subset which occur as
orders of elements of $H_i$ belongs to $L_i^0$. Set $N=\bigcap_{i=1}^{11} N_i$.

It follows from the definition that for a Mukai representation $\rho=(\rho_i)$ the individual
characters $\rho_i$ belong to $L_i^0$. Since the value of the character is independent of $i$ they define
together an element of $N$. 

\smallskip 

Let $K$ be the submodule of $M$ generated by the restriction of the generalized rational characters
of $M_{24}$ to those conjugacy classes of order $1$, $\ldots$, $8$ which are listed
in the table under Remark~\ref{mukaigroups}.2.\ above.
A computer calculation with Magma~\cite{Magma} shows $K=N$. 
Since the Schur indices of the $M_{24}$-representations are all $1$~\cite{Schur}, 
a generalized rational character for $M_{24}$
actually defines a virtual rational $M_{24}$-module. This proves the theorem.
\qed

\begin{rem}
1. \  One  has  that $M/N$ is a finite abelian group abstractly isomorphic to $2 \times 4 \times 24 \times40320$, where we use the short-hand notation $n$ for $\Z/n\Z$.
We also list for each $i$ the possible orders of elements of $H_i$ and the structure of the
quotient modules $N_i/N$ in the following table. 
$$\begin{array}{c|cccccc}
\hbox{no.}       & 1 & 2 & 3 & 4 & 5 & 6\\ \hline
\hbox{orders}    & 1-4,7 & 1-5  & 1-6 & 1-5 & 1-4,8   &1-4,6 \\
 N_i/N           &  4\times 12\times 480  &  4\times 4\times 672   &  2\times 4\times 672   & 12\times 168     & 3\times 420 & 2\times 280 
\end{array}
$$

$$\begin{array}{c|ccccc}
\hbox{no.}       & 7 & 8 & 9 & 10 & 11 \\ \hline
\hbox{orders}    & 1-4,6  & 1-4,6   & 1-4,6   & 1-4  & 1-4,6,8   \\
  N_i/N          &  2\times 840  &  2\times 840   &    2\times 4\times 1120 & 4\times 4\times 3360  & 2\times 1680 
\end{array}  
$$

\noindent 2. \  In addition, the computer calculation shows that in order to obtain the virtual $M_{24}$-representation $\mu$ it 
is enough to only consider the virtual representations $\rho_i$ of the groups $H_1$, $H_5$, $H_6$ and 
in addition one of the three groups $H_2$, $H_3$ or $H_4$.
It is not enough to use only three of the groups $H_i$.
\end{rem}

\smallskip

Let $K'$ and $K''$  be the submodules of $M$ generated by the restriction of the rational characters
of $M_{23}$ respectively $Co_0$ to those conjugacy classes of order $1$,~$\ldots$,~$8$ which are listed
in the table in Remark~\ref{mukaigroups}.2. 
Since $K\subset K'$ and  $K'\subset N$ one has $K'=N$. 
A calculation shows that the lattice $K''$ for the Conway group $Co_0$ has index~$2$ in $K=K'=N$.

\medskip

\begin{cor}\label{cor:virtualM24ellipticgenus}
The complex elliptic genus of a K3 surface can be given the structure of a
virtual rational $M_{24}$-module which is compatible with the $H_i$-module structure
for each of the eleven Mukai groups $H_i$ under restriction.
\end{cor}

\noindent{\bf Proof:} The coefficients of the elliptic genus of a 
K3 surface $X$ are of the form $\chi(X,E)$ where $E$ is a vector bundle 
associated to the tangent bundle of $X$. The $\chi(g;X,E)$ define a
generalized rational $H_i$-character for each of the $H_i$
(see equation~\eqref{eq:fixed} in the last section). 
The result follows now from Theorem~\ref{virtualM24}. 
 \qed

\smallskip

It was shown by Gannon~\cite{Gannon}, that the $M_{24}$-modules of
Mathieu moonshine can be extended to virtual $Co_0$-modules. Thus in these cases
it follows from Theorem~\ref{moonshinegeometric} that
the above mentioned $\Z_2$-ob\-structions in $N/K''$ vanish.


\section{The elliptic genus  of a K3 surface}\label{Holonomy}

We discuss properties of the elliptic genus related to
the hyperk\"ahler geometry of K3 surfaces using vertex algebras.


\subsection{The complex elliptic genus and vertex algebras}\label{ellgenus}

Let $V$ be a vertex operator superalgebra, $M=\bigoplus M_n$ be a (twisted) $V$-module
and $G$ be a subgroup of the automorphism group of $M$ such that
\begin{equation}\label{decomposition}
M=\bigoplus_{\lambda \in {\rm Irr}(G)}\lambda\otimes M_{\lambda}
\end{equation}
as $G \times V^G$-module (i.e.~the $M_{\lambda}$ are trivial $G$-modules) with finite dimensional $\lambda$.
%

For a $G$-principal bundle $P$ over a closed manifold $X$ we have the associated
bundle $\underline{M}=P\times_G M$ over $X$ which at every fiber has the structure of a 
vertex operator superalgebra module isomorphic to $M$.
The grading of $M$ provides the decomposition $\underline{M}=\bigoplus_{n} \underline{M_n}\,q^n$ into finite dimensional
vector bundles over $X$, where we introduced the formal variable $q$.
Let $\underline{W_{\lambda}}$ be the complex vector bundle over $X$ associated to a finite dimensional representation $\lambda$ of $G$.
Then the decomposition~(\ref{decomposition}) of $M$ leads to a decomposition
$$ \underline{M}=\bigoplus_{\lambda \in {\rm Irr}(G)}\underline{W_\lambda}\otimes \underline{M_{\lambda}}$$
where we consider  $\underline{M_{\lambda}}$ as the trivial bundle $X \times  M_{\lambda}$.

On a $d$-dimensional complex manifold $X$, one can consider the Dolbeault operator 
twisted by a holomorphic vector bundle $E$.
The holomorphic Euler characteristic 
$\chi(X,E)=\sum_{q=0}^d(-1)^q \dim H^q(X,E)$ for the sheaf of holomorphic sections in~$E$
equals the index of the Dolbeault operator and can be computed by the twisted
Todd genus
$$\chi(X,E)=  {\rm Td}(X){\rm ch}(E)[X]$$
where ${\rm Td}(X)$ is the total Todd class of $X$ and ${\rm ch}(E)$ is the Chern character of 
$E$; cf.~\cite{Hirzebruch-Habil}. Note that we use the same symbol $\chi(X, E)$ for the virtual vector space and its graded dimension. 
In the following we assume that the bundles $\underline{W_\lambda}$ and hence $\underline{M}$ are holomorphic bundles.

We are interested in the twisted Todd genus 
$$\chi(X,\underline{M})=\sum_n \chi(X,\underline{M_n})\,q^n\in {\Z}[[q]].$$
\begin{thm}\label{character}
The Todd genus $\chi(X,\underline{M})$ is the graded character of a natural virtual vertex operator superalgebra module for $V^G$.
\end{thm}

\noindent{\bf Proof:} One has the following identity between virtual graded vector spaces:
$$\chi(X,\underline{M})
\ =\ \bigoplus_{\lambda \in {\rm Irr}(G)} \chi(X,\underline{W_\lambda}\otimes \underline{M_{\lambda}})
\ =\ \bigoplus_{\lambda \in {\rm Irr}(G)} \chi(X, \underline{W_\lambda})\otimes M_{\lambda}.$$
The first equality is trivial.
The last equality needs some justification. 
We have the diagram 
$$\begin{array}{ccc}
 \Gamma(X,A^{q}\otimes\underline{W_\lambda})\otimes M_{\lambda} & \longrightarrow &  \Gamma(X,A^{q}\otimes \underline{W_\lambda}\otimes \underline{M_\lambda})\\
\\
 \downarrow { \scriptsize \partial^{q}_{\underline{W_\lambda} }\otimes {\rm id }}  &   &  \downarrow { \scriptsize \partial^{q}_{\underline{W_\lambda}\otimes \underline{M_\lambda} }} \\
 \\
 \Gamma(X,A^{q+1}\otimes \underline{W_\lambda})\otimes M_{\lambda}  & \longrightarrow  & \Gamma(X,A^{q+1}\otimes \underline{W_\lambda}\otimes \underline{M_\lambda})
\end{array}$$
where $A^{q}$ is the exterior bundle $\Lambda^{q}T_\C^* X$,
$\partial^{q}_E$ denotes the twisted Dolbeault operator with coefficients in a holomorphic bundle $E$, and the horizontal
arrows are  isomorphisms  induced by the constant sections in the trivial
bundle  $\underline{M_\lambda}$.
A $V^G$-module structure on all spaces is obtained by using constant sections in the bundle $\underline{V^G}$ and the fiberwise action on the last tensor factor.

For a local holomorphic frame $(\Phi_j)$ of a bundle $E$ and a form $\alpha$, the operator $\partial^{q}_E$ is defined by
$ \partial^{q}_E(\alpha\otimes \Phi_j)=  (\partial^{q}\alpha)\otimes \Phi_j$. Thus we see that the above diagram is commutative in the sense of
$V^G$-modules and we have the induced $V^G$-module maps on the level of cohomology. 
\qed

\bigskip

For the complex elliptic genus we will take  
$$V=
\left(
\bigotimes_{n \in \Z_{\geq 0}+\frac{1}{2}}(\Lambda_{yq^n}({\C}^*)\otimes \Lambda_{y^{-1}q^n}( {\C}) )
 \otimes 
\bigotimes_{n=1}^\infty S_{q^n}({\C}\oplus   {\C})
 \right)^{\otimes d}$$
and 
$$M=\left(
\Lambda_{y}( {\bf C}^*) \otimes 
\bigotimes_{n =1}^\infty(\Lambda_{yq^n}({\C}^*)\otimes \Lambda_{y^{-1}q^n}( {\C})) \otimes 
\bigotimes_{n=1}^\infty S_{q^n}({\C}\oplus  {\C}^*) 
 )\right)^{\otimes d}$$
as vertex operator superalgebra and module, resp., for $G$ we take $GL(d,\C)<{\rm Aut}(V)$ and for $P$ we take the holomorphic 
$GL(d,\C)$-principal bundle associated to the tangent bundle of the complex manifold $X$.

Thus $\chi_y(q,{\cal L}X):=(-y)^{-d/2}\chi(X,\underline{M})$ has the structure of a virtual $V^{GL(d,\C)}$-module if we verify
our assumptions on $M$.
If the structure group of $P$ is reduced to a smaller subgroup $G$ of $GL(d,\C)$ --- for example if $X$ is a complex
manifold with trivial canonical line bundle ---  $\chi_y(q,{\cal L}X)$ will be a virtual module
for the larger vertex operator superalgebra $V^G\supset V^{GL(d,\C)}$.

\smallskip

\begin{rem} 1.~In a differential geometric setting it was shown by 
H.~Tamanoi~\cite{Ta1,Ta2} that the level two elliptic genus of spin manifolds is a vertex operator algebra module over 
the vertex operator superalgebra formed by the parallel sections in a certain vertex superalgebra bundle $\underline{V}$ 
which can be identified with $V^G$ where $G$ is the holonomy group.

2.~In~\cite{BoLi} it was observed that the complex elliptic genus of complex manifolds is 
the graded dimension of a vertex algebra defined as the cohomology of the chiral de Rham complex~\cite{MSV}.
It was shown in~\cite{MSV} that a global holomorphic volume form defines an $\N2$ superconformal structure on the space of sections of the chiral de Rham complex and its cohomology. In~\cite{BHS} the authors considered hyperk\"ahler manifolds and showed 
that the chiral de Rham complex has an $\N4$ superconformal structure of central
charge $6d$ where $d$ is the complex dimension of the manifold. 
For the elliptic genus, this has also been investigated in the preprint~\cite{Zhou}.
In~\cite{H2} it was shown that this superconformal 
structure can be decomposed into two commuting $\N4$ superconformal structures of central charge $3d$. In both cases
the superconformal structures depend on the hyperk\"ahler metric. 
It is not clear if these superconformal structures can be extended to a vertex algebra isomorphic to $V^G$.
For the character considerations in the present
paper, it is sufficient to consider only the above established $V^G$-module structure for the elliptic genus.
\end{rem}

%

\bigskip

Finally, assume that a compact group $H$ acts on $X$ by automorphisms compatible with the reduction of the structure group to $G$. 
Then   $\chi_y(q,{\cal L}X)$ is also a virtual $H$-module. More precisely, one has
$$\chi(X,\underline{M})=  \bigoplus_{\lambda \in {\rm Irr}(G)} \chi(X,\underline{W_\lambda})\otimes \underline{M_{\lambda}}$$
as $H\times V^G$-module. Thus for the computation of the equivariant elliptic genus $\chi(g;X,\underline{M})$ 
for $g \in H$ it is enough to know $\chi(g;X,\underline{W_\lambda})$ for all $\lambda \in {\rm Irr}(G)$.

\subsection{The algebra $V^G$}\label{VG}

We restrict now to the case of a K3 surface so that the complex dimension is $d=2$ and the structure group can be reduced to
$G={\rm SL}(2,\C)$. We will describe the vertex algebra $V$ and the fixed point subalgebra $V^G=V^{{\rm SU}(2)}$ in more
detail.

\medskip

The vertex algebra $V$ is generated by fermionic fields $b_1(z)$, $c_1(z)$, $b_2(z)$, $c_2(z)$ and bosonic fields 
$x_1(z)$, $y_1(z)$, $x_2(z)$, $y_2(z)$ with non-regular operator products
\begin{equation*}
 b_i(z)c_j(w)\ \sim \ \frac{\delta_{i,j}}{(z-w)},\qquad x_i(z)y_j(w)\ \sim \ \frac{\delta_{i,j}}{(z-w)^2}.
\end{equation*}
The vectors $b(z)=(b_1(z),b_2(z))^t$, $x(z)=(x_1(z),x_2(z))^t$ carry the standard representation of ${\rm SU}(2)$, while
$c(z)=(c_1(z),c_2(z))$, $y(z)=(y_1(z),y_2(z))$ carry its conjugate representation. 
Define the basic invariant fields
\begin{equation*}
 \begin{split}
 \tr_{n,m}(v,w)&=:\partial^n v_1(z) \partial^m w_1(z):+:\partial^n v_2(z) \partial^m w_2(z):,\\
{\det}_{n,m}(v,w)&=:\partial^nv_1(z) \partial^m w_2(z):-:\partial^nv_2(z) \partial^m w_1(z):,
 \end{split}
\end{equation*}
where $v_i$ and $w_j$ are the components of vectors $v$ and $w$ in $\{ b, c, x, y\}$. 
\begin{lem}\label{lem:basic}
$V^G$ is strongly generated by the basic invariant fields, i.e. $\{{\det}_{n,m}(b,b),{\det}_{n,m}(b,x),{\det}_{n,m}(x,x), {\det}_{n,m}(c,c), {\det}_{n,m}(c,y),{\det}_{n,m}(y,y),
\newline \tr_{n,m}(b,c),\tr_{n,m}(b,y),\tr_{n,m}(x,c),\tr_{n,m}(x,y)\mid n,m\in\mathbb Z_{\geq0}\}$.
\end{lem}
\noindent{\bf Proof:}
Weyl's fundamental theorem of invariant theory for the fundamental representation of ${\rm SU}(N)$ 
\cite{We} states that all invariants can be expressed in
terms of the basic invariants, i.e.\ traces and determinants. 
Hence especially every element of $V^G$ can be expressed as a normally ordered polynomial
in above generators. \qed

\begin{lem}\label{lem:algebras}
$V^G$ contains the following subalgebras:
\begin{enumerate}
\item The simple affine vertex algebra $L_1(sl(2))$ of $sl(2)$ of level one, generated by $\tr_{0,0}(b,c)$, ${\det}_{0,0}(b,b)$, ${\det}_{0,0}(c,c)$.
\item The super Virasoro algebra $\rm{Vir}_6$ of central charge $6$, generated by $\tr_{0,0}(x,y)+(\tr_{1,0}(b,c)-\tr_{0,1}(b,c))/2$.
\item The $\N4$ super Virasoro algebra of central charge $6$, generated by $L_1(sl(2))$,  $\rm{Vir}_6$ and $\tr_{0,0}(x,c)$, $\tr_{0,0}(b,y)$,  ${\det}_{0,0}(b,x)$,  ${\det}_{0,0}(c,y)$.
\end{enumerate}
 \end{lem}
\noindent{\bf Proof:}
These are well-known statements: $\tr_{0,0}(b,c)$, ${\det}_{0,0}(b,b)$, ${\det}_{0,0}(c,c)$ generate a rank one lattice vertex algebra
isomorphic to the simple affine vertex algebra $L_1(sl(2))$ of $sl(2)$ of level one.
The realization of the $\N4$ super Virasoro algebra in terms of these free fields appeared in \cite{ABD}. \qed

\begin{rem}\rm
The universal $\N4$ super Virasoro algebra at central charge $6$ is the quantum Hamiltonian reduction for a non-principal nilpotent element of the universal 
affine vertex superalgebra
$V_{-2}(psl(2|2))$ of $psl(2|2)$ at level minus two. This is a non-generic position and hence the reduced algebra is not simple \cite{HS}.
To our knowledge it is not proven whether the subalgebra of $V$ generated by the fields of point 3 of Lemma \ref{lem:algebras}  describe the simple quotient of the $\N4$ super Virasoro algebra at central charge $6$.
\end{rem}
\begin{rem}\rm
The conformal dimension $\Delta$ with respect to the Virasoro field $\tr_{0,0}(x,y)+(\tr_{1,0}(b,c)-\tr_{0,1}(b,c))/2$ are $\Delta(b_i)=\Delta(c_i)=1/2$ and 
$\Delta(x_i)=\Delta(y_i)=1$. So that the basic invariants have conformal dimension $\Delta(\tr_{n,m}(v,w))= \Delta(\det_{n,m}(v,w))=n+m+\Delta(v)+\Delta(w)$.
\end{rem}
\noindent We define the set 
\begin{equation*}
\begin{split}
S=\{&{\det}_{0,0}(b,b),{\det}_{0,0}(b,x), {\det}_{0,0}(c,c), {\det}_{0,0}(c,y),
\tr_{0,0}(b,c),\tr_{0,0}(b,y),\\
&\tr_{0,0}(x,c),\tr_{0,0}(x,y),
{\det}_{1,0}(b,x),{\det}_{1,0}(x,x), {\det}_{1,0}(c,y),{\det}_{1,0}(y,y),\\
&\tr_{1,0}(b,c),\tr_{1,0}(b,y),\tr_{1,0}(x,c),\tr_{0,0}(x,y)\}  
\end{split}
\end{equation*}
of $16$ fields containing the eight strong generators of the $\N4$ super Virasoro algebra and eight additional 
fields of conformal dimension $2$, $5/2$, $5/2$, $5/2$, $5/2$, $3$, $3$, $3$.
\begin{lem}\label{lem:rel}
Every basic invariant is a normally ordered polynomial of the fields in $S$ and their derivatives. 
\end{lem}
\noindent{\bf Proof:}
The idea of the proof is as for Theorem~7.1 of \cite{CL}, which uses \cite{FKRW}. 
It is well-known (see e.g.\ Section~7 of \cite{CL}), that $\tr_{0,0}(b,c)$ and $\tr_{0,1}(b,c)$ strongly
generate the simple $\mathcal W_{1+\infty}$ algebra of central charge two. Especially, every $\tr_{n,m}(b,c)$
is expressible as a normally ordered polynomial in $\tr_{0,0}(b,c)$ and $\tr_{0,1}(b,c)$ and their derivatives:
\begin{equation*}
 \tr_{n,m}(b,c)=P_{n,m}(\tr_{0,0}(b,c),\tr_{0,1}(b,c)). 
\end{equation*} 
Let $\circ_m$ denote the $m$-th order pole contribution in the operator product algebra.
Then
\begin{align*}\label{eq:firstpole}
 \tr_{0,0}(x,c)\circ_1 \tr_{0,n}(b,c)&=\tr_{0,n}(x,c),  &\hspace*{-1.3cm}{\det}_{0,0}(b,x)\circ_1 \tr_{0,n}(x,c)&={\det}_{0,n}(x,x),\\
\tr_{0,0}(y,b)\circ_1 \tr_{n,0}(b,c)&=-\tr_{0,n}(y,b),  &\hspace*{-1.3cm}{\det}_{0,0}(c,y)\circ_1 \tr_{0,n}(y,b)&={\det}_{0,n}(y,y),\\
{\det}_{0,0}(b,b)\circ_1 \tr_{n,0}(b,c)&=-2 {\det}_{0,n}(b,b), &\hspace*{-1.3cm}{\det}_{0,0}(c,c)\circ_1 \tr_{0,n}(b,c)&=2 {\det}_{0,n}(c,c),\\
{\det}_{0,0}(b,b)\circ_1 \tr_{0,n}(x,c)&= -2{\det}_{0,n}(x,b), &\hspace*{-1.3cm}{\det}_{0,0}(c,c)\circ_1 \tr_{0,n}(y,b)&= -2{\det}_{0,n}(y,c),\\
\tr_{0,0}(x,c)\circ_1 \tr_{0,n}(y,b)&= \tr_{0,n}(y,x)+\tr_{1,n}(c,b).
\end{align*}
We consider the example of $\tr_{m,n}(x,c)$ and $\det_{m,n}(x,x)$. 
It follows from above the equation that
\begin{equation*}
\begin{split}
\tr_{0,n}(x,c)&=  \tr_{0,0}(x,c)\circ_1 \tr_{0,n}(b,c)=\tr_{0,0}(x,c)\circ_1P_{0,n}(\tr_{0,0}(b,c),\tr_{0,1}(b,c))\\
&=Q_{0,n}(\tr_{0,0}(b,c),\tr_{0,0}(x,c),\tr_{0,1}(b,c),\tr_{0,1}(x,c))
\end{split}
\end{equation*}
for some polynomial $Q_{0,n}$ in $tr_{0,0}(b,c)$, $\tr_{0,0}(x,c)$, $\tr_{0,1}(b,c)$, $\tr_{0,1}(x,c)$ and their derivatives. 
{}From this we get 
\begin{equation*}
\begin{split}
{\det}_{0,n}(x,x)&=  {\det}_{0,0}(b,x)\circ_1 \tr_{0,n}(x,c)\\
&={\det}_{0,0}(b,x)\circ_1Q_{0,n}(\tr_{0,0}(b,c),\tr_{0,0}(x,c),\tr_{0,1}(b,c),\tr_{0,1}(x,c))\\
&=R_{0,n}(\tr_{0,0}(b,c),\tr_{0,0}(x,c),{\det}_{0,0}(b,x),\tr_{0,1}(b,c),\tr_{0,1}(x,c),{\det}_{0,1}(b,x))
\end{split}
\end{equation*}
for some polynomial $R_{0,n}$ in $\tr_{0,0}(b,c)$, $\tr_{0,0}(x,c)$, ${\det}_{0,0}(b,x)$, $\tr_{0,1}(b,c)$, $\tr_{0,1}(x,c)$, ${\det}_{0,1}(b,x)$ 
and their derivatives. The claim for $\tr_{m,n}(x,c)$ and $\det_{m,n}(x,x)$ follows by induction to $N=n+m$, 
since every $\tr_{n,m}(v,w)$ (respectively ${\det}_{n,m}(v,w)$) is a linear combination of $\tr_{0,n+m}(v,w)$ 
(respectively ${\det}_{0,n+m}(v,w)$)
and the derivatives of 
the $\tr_{n',m'}(v,w)$ (respectively ${\det}_{n',m'}(v,w)$) with $n'+m'<n+m$. 

The cases of the seven remaining types of basic invariants are proven in complete analogy.
\qed

We can thus characterize $V^G$
\begin{thm}
$V^G$ is strongly generated by the $16$ fields of the set $S$.
\end{thm}
\noindent{\bf Proof:}
Lemma \ref{lem:basic} says that every field of $V^G$ is a normally ordered polynomial of the basic invariants
and their derivatives, which
in turn are by Lemma \ref{lem:rel} themselves normally ordered polynomials in the fields of the set $S$
and their derivatives.  
\qed

\medskip

Going back to general dimension $d$ and the corresponding $V$ as at the end of the last subsection one has:
\begin{rem}
\begin{itemize}
\item
For $G={\rm SL}(d,\C)$ the vertex superalgebra $V^G$ contains the $\N2$ super Virasoro  algebra of central charge $3d$ as
subalgebra (special K\"ahler manifolds).
\item For $G={\rm Sp}(d',\C)$ the vertex algebra $V^G$ contains the $\N4$ super Virasoro algebra of central charge $6d'=3d$ as
subalgebra (hyperk\"ahler manifolds).
\end{itemize}
\end{rem}

\subsection{Characters for $V^G$-modules}

In this section, we consider characters for the fixed point algebra $V^G$ for the
case $G={\rm SL}(2,\C)$.

\medskip

We have the decompositions
\[ V= \bigoplus_{N=0}^\infty \lambda_N \otimes V_{\lambda_N}\quad \text{and} \quad
M = \bigoplus_{N=0}^\infty \lambda_N  \otimes M_{\lambda_N},
\]
where $\lambda_N$ is the $N+1$-dimensional irreducible representation of $G$.
The $V_{\lambda_N}$ are modules 
in the Neveu-Schwarz-sector of the vertex superalgebra $V^G$, while 
the $M_{\lambda_N}$ are in the so-called Ramond-sector. 
The two are related by an automorphism of the $\N2$ super Virasoro algebra that
is called spectral flow in physics. 
\begin{prop}
The $V^G$-modules $M_{\lambda_N}$ and $V_{\lambda_N}$ are irreducible.
\end{prop}
\begin{proof}
The mode algebra of any vertex operator super algebra is associative, and $V$ and $M$ are simple $V$-modules so that the statement follows by Theorem 1.1 of \cite{KR} together with Remark 1.1 of \cite{KR}.
\end{proof}
A character of a module $W$ is defined as
\[ 
\text{ch}_W(y;q)=\text{tr}_W(q^{L_0-\frac{1}{4}}y^{J_0}),
\]
where $L_0$ is the zero-mode of the Virasoro subalgebra and $J_0$ the zero-mode of the Heisenberg subalgebra.
Between the characters of $V$ and $M$ one has the identity $\text{ch}_M(y;q)=q^{1/4}y\,\text{ch}_V(yq^{1/2};q)$.
In the case of the $\N4$ super Virasoro algebra, it is believed~\cite{ET} that modules fall into a generic family 
of irreducible modules parameterized by conformal weight $h$, called typical or massive, with character
\[
 \text{ch}_h(y;q)=q^{h-3/8}\frac{\vartheta_3(y;q)^2}{\eta(q)^3}.
\]
For $h=1/4$ the module is not irreducible and the character is given by the one of three atypical or massless modules 
\[ 
 \text{ch}_{1/4}(y;q)=2\,\text{ch}_{1/4,0}(y;q)+\text{ch}_{1/4,1}(y;q)
\]
with
\[
 \text{ch}_{1/4,0}(y;q)= \frac{\vartheta_3(y;q)}{\eta(q)^3}
\sum_{\alpha\in\Z+\frac{1}{2}}\frac{y^{\alpha+1/2}q^{\alpha(\alpha-1)/2}}{1+yq^{\alpha}}.
\]
\begin{rem}
The massive modules and their characters have also been computed by Kac and Wakimoto \cite{KW2} via quantum Hamiltonian reduction. However irreducibility is not proven.
\end{rem}
\begin{thm}\label{holonomydecomp}
The character $\mathrm{ch}_{V_N}(y;q)$ decomposes into a sum  of atypical (massless) characters $\mathrm{ch}_N^a(y;q)$
and a sum of typical (massive characters) $\mathrm{ch}_N^t(y;q)$ of the $\N4$ super Virasoro algebra of central charge six, where
\begin{equation*}
\begin{split}
\mathrm{ch}_{V_N}(y;q)&= \mathrm{ch}_N^a(y;q)+\mathrm{ch}_N^t(y;q)\\
\mathrm{ch}_N^a(y;q)&= \frac{\vartheta_3(y;q)}{\eta(q)^3}\sum_{\alpha\in\mathbb Z+\frac{1}{2}}\frac{y^{\alpha+1/2}q^{\alpha(\alpha-1)/2}}{1+yq^{\alpha}}\times
\begin{cases}-2 &\mathrm{if}\ N=0, \\ \ \ 1 &\mathrm{if}\ N=1, \\ \ \ 0 &\mathrm{else} \end{cases} \\
\mathrm{ch}_N^t(y;q)&= \frac{\vartheta_3(y;q)^2}{\eta(q)^3}(h_{N}(q)-2h_{N+1}(q)+2h_{N+3}(q)-h_{N+4}(q))\\
h_{N+1}(q)&=\frac{1}{\eta(q)^{3}}\sum_{\substack{m,r,s\in\mathbb Z+\frac{1}{2}\\ r,s>0}}(-1)^{r+s+1}
q^{r|m|+s|N-m|+(\mathrm{sgn}(m)r+\mathrm{sgn}(m-N)s)^2/2-N/2}.
\end{split}
\end{equation*}
\end{thm}
Before proving the theorem, we remark: 
\begin{rem}
The Ramond-characters are immediately obtained via $\text{ch}_{M_N}(y;q)=q^{1/4}y\,\text{ch}_{V_N}(yq^{1/2};q)$. 
Let $\text{ch}^R_h(y;q)=q^{h-3/8}\frac{\vartheta_2(y;q)}{\eta(q)^3}$ be the Ramond-character of conformal weight $h$ of the $\N4$ super Virasoro algebra. 
Expanding the $h_N(q)$ one can explicitly compute the multiplicities of the modules, and the result is of the form
\begin{equation*}
\begin{split}
  \mathrm{ch}_{V_N}(y;q) &=\mathrm{ch}_N^a(y;q)+\sum_{h\in\mathbb Z/4} \text{mult}(N,h)\,\text{ch}_h(y;q),\\
 \mathrm{ch}_{M_N}(y;q) &=\mathrm{ch}_N^{R,a}(y;q)+\sum_{h\in\mathbb Z/4} \text{mult}(N,h)\,\text{ch}^R_h(y;q),
\end{split}
\end{equation*}
with the same multiplicities $\text{mult}(N,h)$ in the Ramond- and NS-sector. 
On the other hand, from the explicit form of $M$, it is clear that only states with conformal dimension in $\Z+1/4$ appear implying
$\text{mult}(N,h)\neq 0$ only if $h\in\Z+1/4$.
\end{rem}
\noindent{\bf Proof:}
The proof of this theorem uses the ideas of chapter~8 of~\cite{DMZ}, where
meromorphic Jacobi forms have been decomposed into Appell-Lerch sums and holomorphic Jacobi forms, see also~\cite{BF}.

The graded character of $V$, graded by the natural Virasoro zero-mode $L_0$, by the charge $f$ of the group $\rm{SL}(2,\C)$
and by the zero-mode $J_0$ of the level one affine $sl(2)$ subalgebra is for $0<|q|<1$ and $|q|<|z|<1$
\begin{equation*}
\begin{split}
\text{ch}_V(z,y;q)&= \text{tr}_V(q^{L_0-\frac{1}{4}}z^fy^{J_0})\\
&=q^{-\frac{1}{4}} \prod_{n=1}^\infty \frac{(1+zyq^{n-\frac{1}{2}})(1+zy^{-1}q^{n-\frac{1}{2}})(1+z^{-1}yq^{n-\frac{1}{2}})(1+z^{-1}y^{-1}q^{n-\frac{1}{2}})}{(1-zq^n)^2(1-z^{-1}q^n)^2}\\
&= \frac{\vartheta_3(y;q)^2}{\eta(q)^6}(1-z^{-1})^2\sum_{m,m'\in\mathbb Z}\frac{z^{m+m'}}{(1+yq^{m-\frac{1}{2}})(1+y^{-1}q^{m'-\frac{1}{2}})}.
\end{split}
\end{equation*}
Here, the last identity follows directly from the denominator identity of the affine Lie superalgebra of $sl(2|1)$ \cite{KW1}.
By defining 
$$g_N= \frac{\vartheta_3(y;q)}{\eta(q)^3}\sum_{m\in\mathbb Z}\frac{1}{(1+yq^{m-\frac{1}{2}})(1+y^{-1}q^{N-m-\frac{1}{2}})},$$
the character decomposes as
$$\text{ch}_V(z,y;q)= \frac{\vartheta_3(y;q)}{\eta(q)^3}\sum_{N\in\mathbb Z}z^N(g_N-2g_{N+1}+g_{N+2}).$$
We need to compute the $g_N$ in order to determine the characters of $V^G$. For this consider the decomposition
$$ \text{ch}_V(z,y;q)=\sum_{N=0}^\infty \chi_N\cdot  \text{ch}_{V_N}$$
of the character of $V$ where
$\chi_N=\sum_{i=0}^{N}z^{-N+2i}$ is the character of $\lambda_N$ and $\text{ch}_{V_N}$ the character of $V_{N}$. 
Since the coefficient of $z^N$, $N\geq 0$, contains exactly the contributions for the representations $\lambda_N$, $\lambda_{N+2}$,
$\ldots$, the character $\text{ch}_{V_N}$ equals the difference of the coefficients of $z^N$ and $z^{N+2}$, i.e.\ one has
$$\text{ch}_{V_N}=\frac{\vartheta_3(y;q)}{\eta(q)^3}(g_N-2g_{N+1}+2g_{N+3}-g_{N+4}).$$

\medskip

In order to compute $g_N$ explicitly, we follow \cite{DMZ}. 
The elliptic transformation of $g_N(u;\tau)$ for $y=e^{2\pi i u}$, $q=e^{2\pi i \tau}$ is  
$$ g_N(u+m\tau+n;\tau)=y^{-m}q^{-m^2/2}g_N(u;\tau)$$
meaning that the $g_N$ have elliptic transformation properties as a Jacobi form of index $1/2$.
We define the Fourier part of $g_N$ as follows: let 
$$g_{N,\ell}^{u_0}=q^{-\ell^2/2}\int_{u_0}^{u_0+1}g_N(u;\tau)y^{-\ell}du.$$
Then the elliptic transformation of $g_N$ implies
$$ g_{N,\ell}^{u_0+\tau}=g_{N,\ell+1}^{u_0}$$ 
and hence $h_{N}=g_{N,\ell}^{-\ell\tau}$, $\ell\in\mathbb Z$, is independent of the choice of the integer $\ell$.
Expanding $g_N$ as a power series in $q$ and $y$ allows to compute the integral for $h_{N}$. The result is 
\begin{equation*}
\begin{split}
 h_{N+1}(q)&=\frac{1}{\eta(q)^{3}}\sum_{\substack{m,r,s\in\mathbb Z+\frac{1}{2}\\ r,s>0}}(-1)^{r+s+1}
q^{r|m|+s|N-m|+(\mathrm{sgn}(m)r+\mathrm{sgn}(m-N)s)^2/2-N/2}.
\end{split}
\end{equation*}
The Fourier part of $g_N$ is now
$$g_N^F=\vartheta_3(y;\tau)h_{N}(\tau).$$
The advantage is that this immediately provides a decomposition of characters if $g_N$ is holomorphic, as in that case
$g_N^F=g_N$. Looking at the definition of $g_N$, we see that $g_1$ is meromorphic but not holomorphic, while for all $N\neq 1$ 
the $g_N$ are holomorphic. In summary, we have shown
$$ \mathrm{ch}_N= \frac{\vartheta_3(y;q)^2}{\eta(q)^3}(h_{N}-2h_{N+1}+2h_{N+3}-h_{N+4})$$
for $N\geq 2$.

\smallskip

It remains to compute $g_1$.
The poles of $g_1(y;q)$ are at the points $(-q^{\alpha-1/2};q)$ for integers $\alpha$. 
Its (weighted) residues are related using the elliptic transformations as follows. 
Let $C_\alpha$ be a small contour including only the pole at $y=-q^{\alpha-1/2}$, then we have
\begin{equation*}
\begin{split}
\text{Res}_{y=-q^{\alpha-\frac{1}{2}}}(g_1(y;q)y^{-\ell})&= \frac{1}{2\pi i}\oint_{C_\alpha}y^{-\ell}g_1(y;q)du
= \frac{1}{2\pi i}\oint_{C_0}(yq^\alpha)^{-\ell}g_1(yq^\alpha;q)du\\
&= (-1)^{\ell+\alpha} q^{\frac{1}{2}(\alpha+\ell-2\alpha\ell-\alpha^2)}\text{Res}_{y=-q^{-\frac{1}{2}}}(g_1(y;q))\\
&= \frac{(-1)^{\ell+\alpha}}{2\pi i}q^{\frac{1}{2}(\alpha+\ell-2\alpha\ell-\alpha^2-1/4)}.
\end{split}
\end{equation*}
We define the polar part of $g_N$ as 
$$g_N^P=g_N-g_N^F.$$
Of course $g_N^P=0$ if $N\neq 1$. In order to compute $g_1^P$ we proceed as in the proof of Theorem 8.1 of \cite{DMZ}.
Let $\tau=\tau_1+i\tau_2$ for $\tau_1$, $\tau_2\in\mathbb R$ and define $A$ by $\text{Im}(u)=A\tau_2$ and choose a contour 
as described in Theorem 8.1 of \cite{DMZ}. Further fix a point $Q=A\tau +B$ for $A$, $B\in\mathbb R$ so that $g_1(e^{2\pi i Q},q)$ is finite. Then
\begin{equation*}
\begin{split}
g^P_1(y;q)&=g_1(y;q)-g^F_1(y;q)\\
&= \sum_{\ell\in\Z}(g_{1,\ell}^Q(q)-h_{1}(q))q^{\ell^2/2}y^\ell\\
&= 2\pi i \sum_{\ell\in \Z,\,\alpha\in\Z+\frac{1}{2}}\frac{\text{sgn}(\alpha-A)-\text{sgn}(\alpha+\ell)}{2}\,\text{Res}_{y=-q^\alpha}(g_1(y;q)y^{-\ell})y^\ell\\
&=  \sum_{\ell\in \Z,\,\alpha\in \Z+\frac{1}{2}}\frac{\text{sgn}(\alpha-A)-\text{sgn}(\alpha+\ell)}{2}\,y^\ell(-1)^{\ell+\alpha+1/2} q^{-\alpha^2/2-\alpha\ell}\\
&= \sum_{\alpha\in \Z+ \frac{1}{2}}\frac{(yq^{-\alpha})^{-\alpha+1/2}}{1+yq^{-\alpha}}q^{-\alpha^2/2}
=  \sum_{\alpha\in \Z+\frac{1}{2}}\frac{y^{\alpha+1/2}q^{\alpha(\alpha-1)/2}}{1+yq^{\alpha}}.
\end{split}_{\qed}
\end{equation*}

\medskip

We finally list explicit decompositions of the first few $V^G$-module characters into the first few
$\N4$ characters in Table~\ref{N4-decomp}.

\begin{table}\caption{Decomposition multiplicities of $V^G$-module characters into characters for the $\N4$ super Virasoro algebra}\label{N4-decomp}
\centerline{
$\begin{array}{c|c|rrrrrrrrrrrrrrr}
 & \text{atypical} & \multicolumn{12}{c}{\text{typical $\text{ch}_h(y;q)$ for $h=$ }}\\
& \text{ch}_{1/4,0}(y;q) & 0 & 1 & 2 & 3 & 4 & 5 & 6 & 7 & 8 & 9 & 10 & 11 \\\hline 
V_{\lambda_0} & -2\phantom{-} &  1 & 0 & 1 & 0 & 1 & 3 & 2 & 6 & 11 & 13 & 24 & 43 \\
V_{\lambda_1} & 1 &    0 & 0 & 0 & 2 & 2 & 2 & 8 & 10 & 16 & 30 & 46 & 68 \\
V_{\lambda_2} & 0 &    0 & 1 & 0 & 1 & 3 & 5 & 7 & 14 & 22 & 39 & 60 & 97 \\
V_{\lambda_3} & 0 &    0 & 0 & 2 & 0 & 2 & 6 & 8 & 14 & 28 & 38 & 70 & 112 \\
V_{\lambda_4} & 0 &    0 & 0 & 0 & 3 & 1 & 3 & 9 & 15 & 22 & 45 & 67 & 112 \\
V_{\lambda_5} & 0 &    0 & 0 & 0 & 0 & 4 & 2 & 6 & 12 & 22 & 36 & 66 & 102 \\
V_{\lambda_6} & 0 &    0 & 0 & 0 & 0 & 0 & 5 & 3 & 9 & 18 & 30 & 50 & 95 \\
V_{\lambda_7} & 0 &    0 & 0 & 0 & 0 & 0 & 0 & 6 & 4 & 12 & 24 & 42 & 66 \\
V_{\lambda_8} & 0 &    0 & 0 & 0 & 0 & 0 & 0 & 0 & 7 & 5 & 15 & 30 & 54 \\
V_{\lambda_9} & 0 &    0 & 0 & 0 & 0 & 0 & 0 & 0 & 0 & 8 & 6 & 18 & 36 \\
V_{\lambda_{10}} & 0 &    0 & 0 & 0 & 0 & 0 & 0 & 0 & 0 & 0 & 9 & 7 & 21
\end{array}  $ }
\end{table}


\subsection{Implications for the elliptic genus of a K3 surface}

We can apply now the discussion of the previous subsections to the case of K3 surfaces.

\medskip
Let $V$ be the vertex algebra studied in Subsection~\ref{VG} and let $G={\rm SL}(2,\C)$.

\begin{cor}\label{cor:voa}
The elliptic genus of a K3 surface is the graded character of a natural
virtual vertex operator algebra module for $V^G$ and its $\N4$ super Virasoro vertex subalgebra.
\end{cor}

\noindent{\bf Proof:}
A K3 surface has a trivial canonical line bundle. This allows to reduce the structure group of
the holomorphic tangent bundle from ${\rm GL}(2,\C)$ to ${\rm SL}(2,\C)$. As discussed in
Subsection~\ref{VG}, the vertex algebra $V$ and the $V$-module $M$ can be decomposed into direct
sums of $V^G$-modules. Theorem~\ref{character} and Lemma~\ref{lem:algebras} part 3.\ together with
the description of the elliptic genus as twisted Todd genus give now the stated virtual vertex algebra
module structures.
\qed

\medskip

For the explicit decomposition of the elliptic genus into characters of $V^G$ and the $\N4$ 
super Virasoro vertex algebra one has to compute the holomorphic Euler characteristics
$\chi(X,\underline{\lambda_n})$ where $\lambda_n$ is the $(n+1)$-dimensional irreducible representation
of ${\rm SL}(2,\C)$. This will be done in the final section.


\section{The elliptic genus and Mathieu Moonshine}\label{Mathieumoonshine}

In Section~\ref{Mukaigroups}, we have seen that the elliptic genus
of a K3 surface can be given the structure of a virtual $M_{24}$-module
extending the virtual $H$-module structures for symplectic automorphism
groups $H$.

In Section~\ref{Holonomy}, it was shown that the Calabi-Yau structure of 
K3
 surfaces leads to a $V^{SU(2)}$-module structure for its elliptic
genus. This implies in particular that there is a decomposition of the elliptic genus
into characters of the $\N4$ super Virasoro vertex operator algebra
compatible with the above virtual $M_{24}$-module structures.

Somewhat disappointingly, we will show in this section that the resulting 
virtual $M_{24}\times V^G$-module structures cannot be made compatible with the predicted 
McKay-Thompson series of Mathieu moonshine.


\subsection{Rational Mathieu Moonshine}\label{rationalmoon}

The coefficients of the complex elliptic genus of a complex manifold
are the holomorphic Euler characteristics of certain vector bundles
associated to the tangent bundle of the manifold.
In the case of a K3 surface $X$, the structure group
of the principal bundle corresponding to the tangent bundle can be reduced
to ${\rm SU}(2)$. 
Every complex vector bundle associated to the tangent bundle
of $X$ is as a differentiable vector bundle isomorphic to a direct sum 
of vector bundles associated to irreducible representations of ${\rm SU}(2)$. 
The irreducible representations of ${\rm SU}(2)$ are
exactly the symmetric powers of the two-dimensional defining representation
of ${\rm SU}(2)$. 

We collect all the symmetric powers of the tangent bundle~$T$ of~$X$ 
in the  generating series $S_tT=\bigoplus_{n=0}^\infty S^nT\cdot t^n$ and
investigate the expression
$$ \chi(X,S_tT)=\sum_{n=0}^\infty \chi(X,S^nT)\cdot t^n\in \Z[[t]].$$

\medskip

Let $E$ be a holomorphic vector bundle on a complex manifold $X$. Recall
from Section~\ref{ellgenus} that by the Riemann-Roch theorem one has
$$\chi(X,E)={\rm Td}(X){\rm ch}(E)[X].$$
For the symmetric powers
of the tangent bundle we get
${\rm ch}(S_tT)=\prod_{i=1}^d(1-e^{x_i}t)^{-1}$
where the $x_i$ are the formal roots of the total Chern class of $T$ and $d$ is the complex dimension
of $X$. Thus $ \chi(X,S_tT)$ is the complex genus for the characteristic power series
$$Q(x)=\frac{x}{1-e^{-x}}\cdot \frac{1}{1-e^xt}=
\frac{1}{1-t}\,\left(1-\frac{1+t}{2(t-1)}\, x+\frac{1+10t+t^2}{12(t-1)^2}\,  x^2+\cdots \right)$$
in the sense of~\cite{Hirzebruch-Habil}.

For $X$ a K3 surface, the only non-vanishing Chern number is $c_2[X]=24$ which equals the Euler characteristic
of $X$. Using the methods developed in~\cite{Hirzebruch-Habil}, we easily obtain
\begin{eqnarray}
r_{1A}:=\chi(X,S_tT) & =   & \frac{2 - 28 t+2\,t^2}{(t-1)^4} \\ \nonumber
                          & =  & 2 - 20\,t - 90\,t^2 - 232\,t^3 - 470\,t^4 - 828\,t^5 - 1330\,t^6  + \cdots.
\end{eqnarray}

\medskip
Given a non-trivial symplectic automorphism $g$ of $X$, we can compute the equivariant genus 
$r_{g}(t)=\chi(g;X,S_tT)$ by using the information from Table~\ref{finitesymp}
and applying the equivariant Riemann-Roch~\cite{AtiyahBott} formula:
$$\chi(g;X,S_tT)=\sum_{p_i} \frac{1}{(1- \lambda_i)(1- \lambda_i^{-1})}\cdot \frac{1}{(1-\lambda_i^{-1}\,t)(1-\lambda_i\,t)}.$$
Here, $(\lambda_i,\lambda_i^{-1})$ are the eigenvalues of $g$ acting on the tangent
space of the isolated fixed point $p_i$.

We recall from Section~\ref{Mukaigroups} that a finite symplectic automorphism $g$ of
a K3 surface can be considered as an element of $M_{23}$ or $M_{24}$ via Mukai's theorem.
Up to taking inverses, the associated conjugacy class in $M_{23}$ or $M_{24}$ depends only on the order 
of~$g$. For the so arising non-trivial conjugacy classes of $M_{23}$ 
one obtains the following seven rational functions: 
\begin{eqnarray*}
r_{2A} & = & \frac{2}{\Phi_2(t)^2} \\
r_{3A} & = & \frac{2}{\Phi_3(t)} \\
r_{4A} & = & \frac{2}{\Phi_4(t)} \\
r_{5A} & = & \frac{2+2t+2t^2}{\Phi_5(t)} \\
r_{6A} & = & \frac{2}{\Phi_6(t)} \\
r_{7AB} & = & \frac{2(1+t^4)+3(t+t^3)+4t^2}{\Phi_7(t)} \phantom{111111111111111111111} \\
r_{8A} & = & \frac{2+2t+2t^2}{\Phi_8(t)}.
\end{eqnarray*}
Here, $\Phi_n(t)$ denotes the $n$-th cyclotomic polynomial of degree $\varphi(n)$.

\smallskip 


We like to make $\chi(X,S_tT)$ into a graded virtual rational $M_{23}$- or $M_{24}$-module with
characters $r_{1A}$, $r_{2A}$, $\dots$,  $r_{7AB}$, $r_{8A}$ as above.
By Theorem~\ref{virtualM24}, we know that this is always possible.
For a rational $M_{23}$-module structure, we have to provide
formulas for $r_{g}={\rm tr} (g| \chi(X,S_tT))$ for $g$ belonging to a conjugacy class of $M_{23}$ of the four remaining types
$11AB$, $14AB$, $15AB$ and $23AB$.
We note that for the other eight types of conjugacy classes $[g]$ the rational
function $r_{g}(t)$ has the form $\frac{P(t)}{Q(t)}$ where $Q(t)$ is a power of the $n$-th cyclotomic 
polynomial with $n$ the order of $g$ and $P(t)$ is a palindromic polynomial of degree two less
than the degree of $Q(t)$.
By requiring that the resulting rational functions have the same
structure as for the previous eight classes, we found  the following (non-uniquely determined) expressions:
\begin{eqnarray*}
r_{11AB} & = & \frac{2(1+t^8)+4(t+t^7)+2(t^2+t^6)+(t^3+t^5)+4t^4}{\Phi_{11}(t)} \\
r_{14AB} & = & \frac{2(1+t^4)+(t+t^3)-2t^2}{\Phi_{14}(t)} \\
r_{15AB} & = & \frac{2(1+t^6)+(t+t^5)-3(t^2+t^4)-6t^3}{\Phi_{15}(t)} \\
r_{23AB} & = & \frac{2(1+t^{20})+5(t+t^{19})+7(t^2+t^{18})+5(t^3+t^{17})+5(t^4+t^{16})}{\Phi_{23}(t)}\\
 & &\qquad+\ \frac{-5(t^5+t^{15})-(t^6+t^{14})+13(t^8+t^{12})+6(t^9+t^{11})+15t^{10}}{\Phi_{23}(t)}. 
\end{eqnarray*}
The general formula 
$$r_{NA}=m(N)\sum_{n\in (\Z/N\Z)^*} \frac{1}{(1- e^{2\pi i n/N})(1- e^{-2\pi i n/N})}\cdot \frac{1}{(1-e^{-2\pi i n/N}\,t)(1-e^{2\pi i n/N}\,t)}$$
for $N\leq 8$ where the constant $m(N)$ is chosen such that $r_{NA}=2+O(t)$ cannot be used for $N>8$ although it has the correct shape
since the $t$-expansion contains non-integral coefficients.

The character table of $M_{23}$ allows to compute the corresponding multiplicities $m_\chi(t)$ of the irreducible $M_{23}$-modules
in $\chi(X,S_tT)$. The resulting rational functions in $t$ have integral coefficients, i.e.~$\chi(X,S_tT)$
becomes by taking the above characters $r_{11AB}$, $r_{14AB}$, $r_{15AB}$ and $r_{23AB}$ indeed a virtual
$M_{23}$-module. For example, for the multiplicity of the trivial character one has
$m_{\chi_1}(t)=\frac{P(t)}{Q(t)}$ with 
$P(t)$ a polynomial of degree~$70$ and $Q(t)$ a product of cyclotomic polynomials
of total degree~$72$. 

The partial fraction decompositions of the multiplicity functions $m_\chi(t)$ contain for all
$M_{23}$-characters $\chi$ a dominating term $\frac{c_\chi}{(t-1)^4}$  for
$t\rightarrow 1$ with $c_\chi<0$; for example, $m_{\chi_1}=-\frac{1}{425040}$.
Thus for large enough $n$, the virtual $M_{23}$-modules $-\chi(X,S^nT)$ are 
actual all honest $M_{23}$-modules. We have not been able to make the $-\chi(X,S^nT)$
simultaneously for all $n$ into honest $M_{23}$-modules under the assumptions made.
More precisely, by taking the characters for elements of order $1$ to $8$ to zero but by
allowing arbitrary values for the elements of order $11$, $14$, $15$ and $23$, one is able
to add integral linear combinations of four virtual rational $M_{23}$-modules $\alpha_1$, $\dots$, $\alpha_4$. 
The following table describes the decomposition into irreducibles:
{ \setlength{\arraycolsep}{3.2pt} 
$$  \begin{array}{c|rrrrrrrrrrrrrrrrr}
{\rm irred.} & \chi_1 & \chi_2 & \chi_3 & \chi_4 & \chi_5 & \chi_6 & \chi_7 & \chi_8 & \chi_9 & \chi_{10} & \chi_{11} & \chi_{12} & \chi_{13} & \chi_{14} & \chi_{15} & \chi_{16} & \chi_{17} \\ \hline
 \alpha_1 &     2 & 0 & 2 & 2 & -2 & 0 & 0 & 0 & 0 & 0 & 0 & -1 & -1 & 0 & 0 & 2 & 0 \\
 \alpha_2 &     2 & -2 & 1 & 1 & 2 & 0 & 0 & 0 & -2 & 0 & 0 & 0 & 0 & -1 & -1 & -2 & 2 \\
 \alpha_3 &     2 & -2 & 0 & 0 & 0 & 2 & -1 & -1 & 2 & 0 & 0 & 2 & 2 & 0 & 0 & 0 & -2 \\
 \alpha_4 &     2 & -2 & -2 & -2 & 0 & 2 & 2 & 2 & 0 & -1 & -1 & -2 & -2 & 2 & 2 & 0 & 0
   \end{array}$$}

We list the decomposition of $-\chi(X,S_tT)$ into irreducible virtual $M_{23}$-modules 
for the first $21$ terms in Table~\ref{rationaldecomposition}.
\begin{table}\caption{Decomposition of $-\chi(X,S_tT)$ into irreducible $M_{23}$-modules}\label{rationaldecomposition}
{ \setlength{\arraycolsep}{3.2pt} 
$$  \begin{array}{r|rrrrrrrrrrrrrrrrr}
n & \chi_1 & \chi_2 & \chi_3 & \chi_4 & \chi_5 & \chi_6 & \chi_7 & \chi_8 & \chi_9 & \chi_{10} & \chi_{11} & \chi_{12} & \chi_{13} & \chi_{14} & \chi_{15} & \chi_{16} & \chi_{17} \\ \hline
 0 &   -2 & 0 & 0 & 0 & 0 & 0 & 0 & 0 & 0 & 0 & 0 & 0 & 0 & 0 & 0 & 0 & 0 \\
 1 &     -2 & 1 & 0 & 0 & 0 & 0 & 0 & 0 & 0 & 0 & 0 & 0 & 0 & 0 & 0 & 0 & 0 \\
 2 &     0 & 0 & 1 & 1 & 0 & 0 & 0 & 0 & 0 & 0 & 0 & 0 & 0 & 0 & 0 & 0 & 0 \\
 3 &     2 & -1 & 0 & 0 & 0 & 1 & 0 & 0 & 1 & 0 & 0 & 1 & 1 & 0 & 0 & 0 & -1 \\
 4 &     2 & -2 & -2 & -2 & 0 & 2 & 0 & 0 & 2 & 0 & 0 & 1 & 1 & 1 & 1 & 0 & -2 \\
 5 &     2 & -1 & 0 & 0 & -1 & 1 & -1 & -1 & 2 & 0 & 0 & 1 & 1 & 0 & 0 & 1 & -1 \\
 6 &     0 & 0 & 0 & 0 & 0 & 0 & -1 & -1 & 0 & 0 & 0 & 1 & 1 & 0 & 0 & 0 & 0 \\
 7 &     -1 & 0 & 0 & 0 & 1 & 0 & 0 & 0 & -1 & 0 & 0 & 0 & 0 & 0 & 0 & 0 & 1 \\
 8 &     -2 & 2 & 2 & 2 & 0 & -2 & 0 & 0 & -2 & 1 & 1 & 0 & 0 & -1 & -1 & 0 & 2 \\
 9 &     -1 & 1 & 0 & 0 & 0 & -1 & 2 & 2 & -1 & 0 & 0 & -2 & -2 & 1 & 1 & 1 & 2 \\
 10 &     -2 & 2 & 1 & 1 & 0 & -2 & 0 & 0 & 0 & 1 & 1 & 0 & 0 & 0 & 0 & 0 & 2 \\
 11 &     1 & -1 & -1 & -1 & 1 & 1 & 1 & 1 & 0 & 0 & 0 & 0 & 0 & 1 & 1 & 0 & 2 \\
 12 &     0 & 0 & 0 & 0 & 0 & 0 & -1 & -1 & 0 & 1 & 1 & 2 & 2 & 0 & 0 & 0 & 2 \\
 13 &    2 & -1 & -1 & -1 & 0 & 2 & 1 & 1 & 0 & 0 & 0 & 0 & 0 & 2 & 2 & 2 & 2 \\
 14 &     0 & 0 & 0 & 0 & 0 & 0 & 0 & 0 & 0 & 1 & 1 & 1 & 1 & 2 & 2 & 2 & 2 \\
 15 &     -1 & 0 & -1 & -1 & 1 & 0 & 1 & 1 & 1 & 1 & 1 & 1 & 1 & 2 & 2 & 2 & 3 \\
 16 &     0 & 0 & 1 & 1 & 0 & 0 & 1 & 1 & 0 & 2 & 2 & 1 & 1 & 2 & 2 & 2 & 4 \\
 17 &     -1 & 1 & 1 & 1 & 1 & 0 & 0 & 0 & 1 & 2 & 2 & 3 & 3 & 1 & 1 & 2 & 5 \\
 18 &     0 & 0 & 0 & 0 & 0 & 2 & 0 & 0 & 2 & 2 & 2 & 4 & 4 & 3 & 3 & 2 & 4 \\
 19 &     2 & -1 & 0 & 0 & 0 & 3 & 0 & 0 & 3 & 2 & 2 & 4 & 4 & 3 & 3 & 4 & 5 \\
 20 &     2 & -2 & 0 & 0 & 0 & 2 & 0 & 0 & 2 & 3 & 3 & 4 & 4 & 4 & 4 & 4 & 6
\end{array}
$$}
\end{table}

\medskip

We expect that our rational $M_{23}$ moonshine extends to a rational moonshine for $M_{24}$ or even $Co_0$.
However, because of the non-uniqueness in choosing additional characters and the incompatibility with
Mathieu moonshine discussed in the next subsection, we do not investigate this further.

\subsection{Comparing rational and Mathieu Moonshine}\label{comparing}

We will compare the virtual $M_{23}$-modules structures for the holomorphic Euler 
characteristics $\chi(X,S^nT)$ as in the previous  subsection with Mathieu moonshine.

\medskip

First, we have to verify that for symplectic automorphisms on a K3 surface the equivariant elliptic genus coincides
with the prediction of Mathieu moonshine as in~\cite{GHV2,EH}.

Recall the definition of the $\chi_y$-genus for a $d$-dimensional complex manifold $X$ and
a holomorphic vector bundle $E$
\[
\chi_y(X, E) = \sum_{p=0}^d \chi(X, \Lambda^pT^*\otimes E)\,y^p.
\]
For a compact Lie group $H$ acting on a complex manifold $X$ by automorphisms,
the equivariant elliptic genus 
$$\chi_y(g;q,{\cal L}X):= (-y)^{-d/2}\chi_y\bigl(g;X,\bigotimes_{n=1}^{\infty}\Lambda_{yq^n}T^*
\otimes \bigotimes_{n=1}^{\infty}\Lambda_{y^{-1}q^n}T
\otimes \bigotimes_{n=1}^{\infty}S_{q^n}(T^*\oplus T)\bigr) $$
is an element in $R(H)[y,y^{-1}][[q]]$ and can also be interpreted for fixed $g\in H$ of finite
order as a Jacobi form for a certain congruence subgroup. 
The holomorphic Lefschetz fixed point formula~\cite{AtiyahBott} allows to determine 
the equivariant elliptic genus from information on the fixed point set and the
action on its normal bundle.

\smallskip

For a K3 surface, we can use Table~\ref{finitesymp} to compute the equivariant elliptic genus 
for finite symplectic automorphisms similar as done for $\chi(g;X,S_tT)$ in the previous subsection.

Let $q=e^{2\pi i \tau}$, $y=e^{2\pi i u}$ for $\tau$ in the upper half of the complex plane, and $u\in\C$, and consider the following 
quotient of a standard Jacobi theta function and Dedekind's eta function
\begin{equation*}
\varphi(u;\,\tau):=\vartheta_1(u;\,\tau)\eta(\tau)^{-3} = -i(y^{1/2}-y^{-1/2})\prod_{n=1}^\infty  (1-yq^n)(1-y^{-1}q^{n})(1-q^n)^{-2}.
\end{equation*}
It has the following modular and elliptic transformation properties
\begin{equation*}
 \begin{split}
  \varphi\Bigl(\frac{u}{ct+d};\,\frac{a\tau+b}{c\tau+d}\Bigr)&\,=\,(c\tau+d)^{-1}e^{\frac{\pi i cu^2}{c\tau+d}}\varphi(u;\,\tau),\qquad 
\left({a \ b \atop c \ d}\right)\in {\rm SL}(2,\Z)\\
\varphi(u+\lambda\tau+\mu;\,\tau)&\,=\,e^{-\pi i (\lambda^2\tau+2\lambda u)}(-1)^{\lambda+\mu}\varphi(u;\,\tau),\qquad \lambda,\, \mu\in\Z. 
 \end{split}
\end{equation*}
\begin{defi}
A weak Jacobi form $J$ of weight $k$ and index $m$ for some subgroup of 
$\Gamma\subset {\rm SL}(2,\Z)$ is a holomorphic function on $\H\times\C$ with no poles at the cusp $q=0$, satisfying 
\begin{equation*}
 \begin{split}
  J\Bigl(\frac{u}{ct+d};\,\frac{a\tau+b}{c\tau+d}\Bigr)&=(c\tau+d)^ke^{\frac{2\pi i mcu^2}{c\tau+d}}J(u;\,\tau),\qquad 
\left({a \ b \atop c \ d}\right)\in \Gamma\subset {\rm SL}(2,\Z)\\
J(u+\lambda\tau+\mu;\,\tau)&=e^{-2\pi i m(\lambda^2\tau+2\lambda u)}J(u;\,\tau),\qquad \lambda,\, \mu\in\Z. 
 \end{split}
\end{equation*}
\end{defi}
\begin{lem}\label{lem:jacobi}
For a non-trivial finite symplectic automorphism $g$ of a K3 surface $X$ the equivariant elliptic genus
$\chi_{-y}(g;q,{\cal L}X)$ is a weak Jacobi form for $\Gamma_0({\rm ord}(g))$ of weight zero and index one.  
\end{lem}
\noindent{\bf Proof:}
Let $(\lambda_i,\lambda_i^{-1})$ be the eigenvalues of $g$ acting on the tangent
space of a fixed point $p_i$. The fixed point formula gives
\begin{equation}\label{eq:fixed}
\begin{split}
\chi_y(g;q,{\cal L}X) \,= \,
&\sum_{p_i} (-y)^{-1}\frac{(1+y\lambda_i)(1+y\lambda_i^{-1})}{(1-\lambda_i)(1-\lambda_i^{-1})} \times \\
&\times\prod_{n=1}^{\infty} \frac{(1+yq^n\lambda_i)(1+y^{-1}q^n\lambda_i^{-1})(1+yq^n\lambda_i^{-1})(1+y^{-1}q^n\lambda_i)}
{\bigl((1-q^n\lambda_i)(1-q^n\lambda_i^{-1})\bigr)^2}.
\end{split}
\end{equation}
Let $N=\text{ord}(g)$ and define multiplicities $m(N)$ by $m(2)=8$, $m(3)=3$, $m(4)=2$, $m(5)=m(6)=1$ and $m(7)=m(8)=1/2$. 
Then comparing with Table~\ref{finitesymp}, we see that
\begin{equation*}
 \chi_{-y}(g;q,{\cal L}X)=m(N)\sum_{n\in (\Z/N\Z)^*}
\frac{\varphi\bigl(u+\frac{n}{N};\,\tau\bigr)\varphi\bigl(u-\frac{n}{N};\,\tau\bigr)}
{\varphi\bigl(\frac{n}{N};\,\tau\bigr)\varphi\bigl(-\frac{n}{N};\,\tau\bigr)}.
\end{equation*}
Let $\left({a \ b \atop c \ d}\right)$ be in $\Gamma_0(N)$, that means $c\equiv 0 \bmod{N}$,
and it follows that $d\pmod{N}$ is a unit in $\Z/N\Z$. 
Define $u_\pm=u\pm (c\tau +d)n/N$ and $v_\pm=\pm (c\tau +d)n/N$. Then 
\begin{eqnarray*}
&& \frac{\varphi\bigl(\frac{u}{c\tau +d}+\frac{n}{N};\,\frac{a\tau+b}{c\tau+d}\bigr)\varphi\bigl(\frac{u}{c\tau +d}-\frac{n}{N};\,\frac{a\tau+b}{c\tau+d}\bigr)}
{\varphi\bigl(\frac{n}{N};\,\frac{a\tau+b}{c\tau+d}\bigr)\varphi\bigl(-\frac{n}{N};\,\frac{a\tau+b}{c\tau+d}\bigr)}
\ = \ \frac{\varphi\bigl(\frac{u_+}{c\tau +d};\,\frac{a\tau+b}{c\tau+d}\bigr)\varphi\bigl(\frac{u_-}{c\tau +d};\,\frac{a\tau+b}{c\tau+d}\bigr)}
{\varphi\bigl(\frac{v_+}{c\tau +d};\,\frac{a\tau+b}{c\tau+d}\bigr)\varphi\bigl(\frac{v_-}{c\tau +d};\,\frac{a\tau+b}{c\tau+d}\bigr)}\\
&& \ = \ e^{\frac{2\pi i u^2}{c\tau+d}}\,\frac{\varphi(u_+;\,\tau)\varphi(u_-;\,\tau)}
{\varphi(v_+;\,\tau)\varphi\bigl(v_-;\,\tau)}
\ =\ e^{\frac{2\pi i cu^2}{c\tau+d}}\,\frac{\varphi\bigl(u+\frac{dn}{N};\,\tau\bigr)
\varphi\bigl(u-\frac{dn}{N};\,\tau\bigr)}{\varphi\bigl(\frac{dn}{N};\,\tau\bigr)\varphi\bigl(-\frac{dn}{N};\,\tau\bigr)}
\end{eqnarray*}
where the second equality follows from the modular properties of $\varphi$, and the third one from the elliptic ones. Note that contributions 
from numerator and denominator lead to various cancellations. 
Inserting this identity in the expression for the equivariant elliptic genus proves the modular transformation property.
Its elliptic transformation property follows immediately from the elliptic properties of $\varphi$.  \qed

Now we can easily show:
\begin{thm}\label{moonshinegeometric}
For a non-trivial finite symplectic automorphism $g$ acting on a K3 surface $X$
the equivariant elliptic genus and the twining character determined by the McKay-Thompson series of 
Mathieu moonshine agree, i.e.~one has
$$\chi_{-y}(g;q,{\cal L}X)=\frac{e(g)}{12}\,\phi_{0,1}+f_g\,\phi_{-2,1},$$
where $e(g)$ is the character of the $24$-dimensional permutation representation of $M_{24}$ and $f_g$ is the modular form as in equation~(\ref{thompsonseries}).
\end{thm}
\noindent{\bf Proof:}
By Lemma~\ref{lem:jacobi}, the  $g$-equivariant elliptic genus  is a weak Jacobi form of 
weight zero and index one 
for $\Gamma_0(N)$ with $N=\text{ord}(g)$.
By a result of \cite{AI} every such Jacobi form is a linear combination of 
$\phi_{0,1}$ and $h_N\,\phi_{-2,1}$, where the $\phi_{m,1}$ are weak Jacobi forms of weight~$m$ and index~$1$, and 
$h_{N}$  is a modular form of weight~$2$ for $\Gamma_0(N)$. 
For $y=-1$, the equivariant elliptic genus specializes to the equivariant Euler characteristic $e(g;X)$ of $X$ (cf.~\cite{Ho-Diplom}) whereas 
$\phi_{0,1}$ and $\phi_{-2,1}$ have the values~$12$ and $0$, respectively.  One checks that $e(g; X)=e(g)$.
The space of modular forms of weight~$2$ for $\Gamma_0(N)$ is a finite dimensional vector space of dimension at most~$3$.
By considering a base one sees that comparing the first three coefficients of $h_N$ with $f_g$ is completing the proof. \qed

\begin{rem}\rm It was observed by several authors (e.g.,~\cite{C}) that the McKay-Thompson series 
of the Mathieu moonshine observation for those elements $g\in M_{24}$ which can be realized
by symplectic group actions on a K3 surface are identical to the equivariant 
elliptic genus. However, the elliptic genus there is understood as the elliptic
genus of an ${\cal N}\!=\!(4,4)$ super conformal field theory and not in the sense of the topological
defined elliptic genus of a manifold as in the present paper. Without rigorous
definition of ${\cal N}\!=\!(4,4)$ super conformal field theories it is not obvious that 
both definitions are the same.
\end{rem}

\medskip

The sum decompositions of the complex genus of a K3 surface into characters of modules of the
holonomy subalgebra and the $\N4$ super Virasoro algebra 
\begin{eqnarray}
\chi_{-y}(q,{\cal L}X) & = &  (-y)^{-d/2}\chi_y\bigl(X,\bigotimes_{n=1}^{\infty}\Lambda_{yq^n}T^*
\otimes \bigotimes_{n=1}^{\infty}\Lambda_{y^{-1}q^n}T
\otimes \bigotimes_{n=1}^{\infty}S_{q^n}(T^*\oplus T)\bigr) \nonumber \\ 
                   & = & \sum_{n=0}^{\infty} \chi(X,S^nT)\cdot {\rm ch}_{M_n}(y;\,q) \nonumber  \\ 
                   & = &  24\cdot {\rm ch}_{h=1/4,l=0}(y;\,q)+ \sum_{n=0}^\infty A_n  \cdot {\rm ch}_{h=1/4+n}(y;\,q) \label{n4decomp}
\end{eqnarray}
as discussed in Section~\ref{Holonomy}
allows to express the coefficients $A_n$ in terms of Todd genera twisted by symmetric powers.
This is easily achieved for small $n$ with the help of Theorem~\ref{holonomydecomp}.
Alternatively, one can use 
the identities $S^nT=S^n(T^*)$,  $\Lambda_tT=S^0T\oplus S^1T\,t\oplus S^0T\,t^2$,
the Clebsch-Gordan rule $S^iT\otimes S^jT=S^{|i-j|}T\oplus S^{|i-j|+2}T\oplus \cdots \oplus S^{i+j}T$
and the explicit formulas for the  $\N4$ characters ${\rm ch}_{h=1/4,l=0}$ and  ${\rm ch}_{h=1/4+n}$.

The first few coefficients are given by
\begin{eqnarray*}
\sum_{n=0}^\infty A_n q^n & = & -\chi(X,\C)-\chi(X,S^2T)\,q-\chi(X,\C\oplus 2S^3T)\,q^2\\
&& \quad         -\ \chi(X,2 T \oplus  S^2T \oplus  3 S^4T)\,q^3\\
&& \quad         -\ \chi(X,\C \oplus  2  T \oplus  3 S^2T \oplus  2 S^3T \oplus  S^4T \oplus  4 S^5T)\,q^4+\cdots .
\end{eqnarray*}
Here we wrote shortly $nF$ for the $n$-fold direct sum vector bundle $F^{\oplus n}$.

According to  Mathieu moonshine, the $A_n$ can be considered as 
the dimensions of virtual $M_{24}$-modules $K_n$ (actually honest ones for $n>0$)
with explicit given McKay-Thompson series $\Sigma_{g}(q)=q^{-1/8}\sum_{n=0}{\rm Tr}(g|K_n)\,q^n$.

\smallskip

We like to investigate if Mathieu moonshine which uses the characters for the \hbox{$\N4$} super Virasoro algebra 
can be extended to a rational $M_{24}$-moonshine for the symmetric powers by using the characters for the holonomy invariant subalgebra.
In this case the series $r_{g}(t)=\chi(g;X,S_tT)$ can be computed recursively by the above formula given 
$\Sigma_{g}$ and $\chi(g;X,T)$.  
The coefficient $24$ of ${\rm ch}_{h=1/4,l=0}$ in equation~(\ref{n4decomp}) is the Euler characteristic 
of the K3 surface. As $M_{24}$-character it is interpreted in Mathieu moonshine 
as the permutation character $\underline{1}+\underline{23}$ of $M_{24}$ for its natural action on $24$ elements.
This implies that $\chi(g;X,T)$ equals the virtual $M_{24}$-character $3\cdot\underline{1}-\underline{23}$.

{}From Theorem~\ref{moonshinegeometric} it follows that for the $M_{23}$ conjugacy classes $[g]\in \{1A,2A,3A,4A,5A,6A,7AB,8A\}$
we get the rational functions already listed in Section~\ref{rationalmoon}. We checked this for the first ten coefficients.

For the remaining classes $\{11AB,14AB,15AB,23AB\}$ we get:
\begin{eqnarray*}
r_{11AB} & = & 2 + 2\,t - 2\,t^2 - t^3 - \frac{2}{3}\, t^4 - \cdots \\
r_{14AB} & = & 2 + 3\,t - t^2 - t^3 - \frac{5}{3}\,t^4 -\cdots \\
r_{15AB} & = & 2 + 3\,t - \frac{1}{2}\,t^3- 2\,t^4 -  \cdots \\
r_{23AB} & = & 2 + 3\,t + 2\,t^2 - 2\,t^3 - \frac{7}{3}\,t^4- \cdots.
\end{eqnarray*}
Since in each case either the coefficient of $t^3$ or of $t^4$ is not integral, we conclude
that $\chi(X,S_tT)$ cannot have the structure of a virtual $M_{23}$- or $M_{24}$-module compatible 
with the prediction of Mathieu moonshine.

This conclusion even holds if we do not assume that $\chi(g;X,T)=\alpha$ is the virtual $M_{24}$-character
$3\cdot\underline{1}-\underline{23}$. For example one has
$$r_{15AB}=2 + \alpha\,t - \frac{1}{2}\,t^3 - \frac{2\alpha}{3}\,t^4 -\frac{3 + 4\alpha}{12}\,t^5+\cdots$$
and so $\alpha$ and the coefficient of $t^5$ cannot simultaneously be integral.

\smallskip

We also could not find simple rational expressions for those $r_{g}$.
However, for the two $M_{24}$ conjugacy classes $2B$ and $4A$ we found the following
formulas (verified up to order 20):
\begin{eqnarray*}
r_{2B({M_{24}})} & = &\frac{2+8t+2t^2}{\Phi_2(t)^2\Phi_4(t)}, \\
r_{4A({M_{24}})} & = & \frac{2(1+t^5) + 6(t+t^4) + 12(t^2 + t^3)}{\Phi_2(t)\Phi_4(t)\Phi_8(t)}. 
\end{eqnarray*}
For the remaining conjugacy classes of $M_{24}$ not contained in $M_{23}$ we could again not find such
expressions. Also, the $r_{g}$ contain always non-integral coefficients in those cases.

\medskip

Combining this with the results of Section~\ref{Holonomy}, we conclude that the vertex algebra structure arising from
the hyperk\"ahler geometry of K3 surfaces can explain the $\N4$ super Virasoro module structure of the elliptic genus, 
and it can explain the virtual $M_{24}$-module structure.  However, there is no $M_{24}$-moonshine
for the decomposition of $\chi_y(q,{\cal L}X)$ into characters of $V^G$-modules compatible with 
the prediction of Mathieu moonshine.

\begin{rem}
Call the $\N4$ super Virasoro algebra $\text{Vir}^{\mathcal N=4}$.
It is natural to ask if the virtual $H_i\times \text{Vir}^{\mathcal N=4}$-module structure of $\chi_y(q, \mathcal L X)$ can be extended to a $M_{24}\times \text{Vir}^{\mathcal N=4}$-module structure. 
On the level of characters this is the case due to Theorems \ref{moonshinegeometric} and \ref{thm:gannon}.  
\end{rem}

One can also ask if the McKay-Thompson series of Mathieu moonshine can be modified to make
them compatible with the $V^G$-module structure of the elliptic genus. Our $r_g$ for 
$g$ of type  $11AB$, $14AB$, $15AB$, $23AB$ cannot be described by weight~$2$ modular forms $h_g$ for 
$\Gamma_0(n)$, $n={\rm ord}(g)$, in the form explained in the proof of Theorem~\ref{moonshinegeometric}.
It is however at least sometimes possible to obtain rational expressions for $r_g$ for certain 
modular forms $h_g\in M_2(\Gamma_0(n))$, for example one can realize
$$\tilde{r}_{11AB}=\frac{2(1+t^8)+4(t+t^7)+\frac{32}{5}(t^2+t^6)+\frac{38}{5}(t^3+t^5)+\frac{42}{5}t^4}{\Phi_{11}(t)}.$$
This gives however not a $M_{23}$- or $M_{24}$-character.

The McKay-Thompson series for a set of modular forms $f_g\in M_2(\Gamma_0(N))$ for $g$ a conjugacy class in $M_{24}$ 
describe a graded virtual $M_{24}$-modules if the coefficients of the $f_g$ satisfy certain divisibility conditions
which in turn can easily be described by the integral structure of $M_2(\Gamma_0(N))$; cf.~\cite{Gannon}. 
In particular, the condition to define a graded virtual $M_{24}$-module alone does not determine the $f_g$ uniquely.


We recall some arguments why the predicted McKay-Thompson series of Mathieu moonshine 
as in~\cite{GHV2,EH} should be the correct choice: The first is that the $K_n$ for $n>0$ are
indeed honest $M_{24}$-modules which decompose for the first few $n$ into irreducibles with
very small multiplicity (in fact only one of two irreducibles are present for $1\leq n \leq 5$).
Secondly, Gaberdiel,  Hohenegger, and Volpato determined in~\cite{GHV3} the expected structure of 
symplectic like symmetry groups of ${\cal N}\!=\!(4,4)$ super conformal field theories.
This allows to realize additional (but not all!) conjugacy classes of $M_{24}$ as symmetries of such
conformal field theories. For some of those they showed that the equivariant elliptic genus in the
sense of conformal field theory equals the one predicted by Mathieu moonshine.
Finally, Cheng and Duncan gave in~\cite{CD} a description of the corresponding Mock modular forms in terms 
of Rademacher sums (thought as an analogy of the genus zero property of the McKay-Thompson series of Monstrous moonshine)
distinguishing the chosen $f_g$ of Mathieu moonshine.

\medskip

We conclude with the remark that a refined geometric approach should consider group actions on K3 surfaces
with generalized Calabi-Yau structures (cf.~\cite{Hu}), the corresponding generalization of the
chiral De Rham complex~\cite{HelZab1} and the additional super Virasoro module structures coming from
generalized Calabi-Yau metric structures~\cite{HelZab2}.
However, it may be too much to expect that this approach alone can explain Mathieu moonshine.


\begin{thebibliography}{ABKS}

\bibitem[AB]{AtiyahBott}  M. F. Atiyah and R.  Bott, \textit{A Lefschetz fixed point formula 
for elliptic complexes I,} Ann. of Math. (2) {\bf 86} (1967) 374--407.

\bibitem[ABD]{ABD}
  M.~Ademollo, L.~Brink, A.~D'Adda, R.~D'Auria, E.~Napolitano, S.~Sciuto, E.~Del Giudice and P.~Di Vecchia {\it et al.},
  \textit{Dual String with U(1) Color Symmetry},
  Nucl.\ Phys.\ B {\bf 111} (1976) 77.

\bibitem[AI]{AI}
H. Aoki and T. Ibukiyama, \textit{Simple graded rings of Siegel modular
orms, differential operators and Borcherds products}, Int. J. Math.
{\bf 16} (2005), 249--279.

\bibitem[AT]{atlas}
Conway, J. H.; Curtis, R. T.; Norton, S. P.; Parker, R. A.; Wilson, R. A. 
{\it Atlas of finite groups. Maximal subgroups and ordinary characters for simple groups,}
With computational assistance from J. G. Thackray. Oxford University Press, Eynsham, 1985. 
xxxiv+252 pp.

\bibitem[Be]{Schur} 
M.~Benard {\it Schur indexes of sporadic simple groups,}
J. Algebra {\bf 58} (1979), 508--522.

\bibitem[Bo]{B} R. E. Borcherds, \textit{Monstrous Moonshine and Monstrous Lie Superalgebras}, 
Invent. Math. {\bf 109} (1992), 405--444.

\bibitem[BF]{BF}
K. Bringmann, A. Folsom, \textit{On the asymptotic behavior of {K}ac-{W}akimoto characters},
   Proc. Amer. Math. Soc. {\bf 141} (2013) 1567--1576. 

\bibitem[BHS]{BHS}D. Ben-Zvi, R. Heluani and  M. Szczesny,   \textit{Supersymmetry of the chiral de Rham complex},
 Compos. Math.  {\bf 144}  (2008),  503--521.

\bibitem[BL]{BoLi}  L. A. Borisov and A. Libgober, 
\textit{Elliptic genera of singular varieties, orbifold elliptic genus and chiral de Rham complex},
 Mirror symmetry, IV (Montreal, QC, 2000),  325--342, AMS/IP Stud. Adv. Math., 33, 
Amer. Math. Soc., Providence, RI, 2002. 

\bibitem[C]{C}
M.~C.~N.~Cheng,
\textit{K3 Surfaces, $N=4$ Dyons, and the Mathieu Group $M_{24}$},
Comm.\ Num.\ Theor.\ Phys.\  {\bf 4} (2010) 623--657.

\bibitem[CD]{CD}
M.\ C.\ N.~Cheng and J.\ F.\ R.\ Duncan,
\textit{On Rademacher Sums, the Largest Mathieu Group, and the Holographic Modularity of Moonshine,}
preprint (2011), arXiv:1110.3859.

\bibitem[CL]{CL} T. Creutzig, and A. R. Linshaw, 
\textit{The super $\mathcal W_{1+\infty}$ algebra with positive integral central charge}, to be published in Transactions of the AMS, arXiv:1209.6032.

\bibitem[CN]{CN} J. H. Conway, and S. P. Norton, \textit{Monstrous moonshine},
\textit{Bull. London Math. Soc.}, {\bf 11} (1979), 308--339.

\bibitem[DMZ]{DMZ} A. Dabholkar, S. Murthy, and D. Zagier, \textit{Quantum Black Holes, Wall Crossing, and Mock Modular Forms}, 
arXiv:1208.4074. 


\bibitem[EH]{EH}
T.~Eguchi and K.~Hikami, 
\textit{Note on twisted elliptic genus of $K3$ surface,} 
{Phys.~Lett.~B} {\bf 694} (2011), 446--455,  arXiv:1008.4924. 

\bibitem[EOT]{EOT} T.~Eguchi, H.~Ooguri and Y.~Tachikawa, \textit{Notes on the K3 Surface and the Mathieu group $M_{24}$},
  Exper.\ Math.\  {\bf 20} (2011) 91--96.

\bibitem[ET]{ET} T. Eguchi, and A. Taormina, \textit{On the unitary representations of $N = 2$ and $N = 4$
superconformal algebras}, Phys. Lett. B {\bf 210} (1988), 125--132.

\bibitem[FKRW]{FKRW} E. Frenkel, V. Kac, A. Radul, and W. Wang, \textit{$\mathcal W_{1+\infty}$ and 
$\mathcal W(\mathfrak g\mathfrak l_N)$ with central charge $N$}, Commun. Math. Phys. {\bf 170} (1995), 337--357.
 
\bibitem[G]{Gannon} T. Gannon,
\textit{Much ado about Mathieu}, arXiv:1211.5531 [math.RT].
  
\bibitem[GHV]{GHV}
M.~R.~Gaberdiel, S.~Hohenegger and R.~Volpato, 
{\textit Mathieu twining characters for $K3$,} 
{ J.~High Energy Phys.} (2010), no. 9, {\bf 058}, 20 pp, 
arXiv:1006.0221. 

\bibitem[GHV2]{GHV2}
M.~R.~Gaberdiel, S.~Hohenegger and R.~Volpato,
\textit{Mathieu Moonshine in the elliptic genus of K3},
JHEP {\bf 1010} (2010) 062, 24 pp.

\bibitem[GHV2]{GHV3} M.~R.~Gaberdiel, S.~Hohenegger and R.~Volpato,
\textit{Symmetries of K3 sigma models},
Commun. Num. Theor. Phys. {\bf 6} (2012) 1--50.

\bibitem[Ha]{Hashimoto} K. Hashimoto,  \textit{Finite symplectic actions on the 
K3 lattice,} Nagoya Math. J. {\bf 206} (2012), 99--153.

\bibitem[He]{H2}
R. Heluani, \textit{Supersymmetry of the chiral de Rham complex: 2 Commuting sectors},
Int. Math. Res. Not. IMRN  2009, No. 6, 953--987.

\bibitem[HZ1]{HelZab1}
R. Heluani, and M. Zabzine, Maxim,  \textit{Generalized Calabi-Yau manifolds and the chiral de Rham complex},
 Adv. Math.  {\bf 223}  (2010),  1815--1844.

\bibitem[HZ2]{HelZab2}
R. Heluani, amd M. Zabzine,  \textit{Superconformal structures on generalized Calabi-Yau metric manifolds},
 Comm. Math. Phys.  {\bf 306}  (2011), 333--364.

\bibitem[Hi]{Hirzebruch-Habil} F.\ Hirzebruch, \textit{Topological methods in algebraic geometry.}
 Third enlarged edition. 
New appendix and translation from the second German edition by R. L. E. Schwarzenberger,
 with an additional section by A. Borel. Die Grundlehren der Mathematischen Wissenschaften,
 Band 131 Springer-Verlag New York, Inc., New York 1966 x+232 pp. 

\bibitem[Ho]{Ho-Diplom} G.~H\"ohn, \textit{Komplexe elliptische Geschlechter und $S^1$-\"aquivariante Kobordismustheorie,}
Diploma thesis, Bonn 1991, arXiv:math/0405232.

\bibitem[HS]{HS} C. Hoyt, and S. Reif, \textit{Simplicity of vacuum modules over affine {L}ie superalgebras},
J. Algebra {\bf 321} (2009), 2861--2874.

\bibitem[Hu]{Hu} D. Huybrechts, \textit{Generalized Calabi-Yau structures, K3 surfaces, and B-fields},
Internat. J. Math. {\bf 16} (2005), 13--36.

\bibitem[Kr]{Krichever} I. M. Krichever,
\textit{Generalized elliptic genera and Baker-Akhiezer functions,} (Russian) Mat. Zametki {\bf 47} (1990), 34--45; 
translation in Math. Notes {\bf 47} (1990), 132--142.

\bibitem[KR]{KR} V. Kac and A. Radul, \textit{Representation theory of the vertex algebra $\mathcal{W}_{1+\infty}$}, Transf. Groups, Vol 1 (1996) 41-70.

\bibitem[KW1]{KW1}
V. Kac, and M. Wakimoto, \textit{Integrable Highest Weight Modules over Affine Superalgebras and
  Number Theory}, Progr. Math. {\bf 123} (1994), 415--456.

\bibitem[KW2]{KW2}
V. Kac, and M. Wakimoto, \textit{Quantum Reduction and
Representation Theory of Superconformal Algebras}, Adv. in Math. 185 (2004), no. 2, 400--458.

\bibitem[Ma]{Magma}
W. Bosma, J. Cannon, and C. Playoust, \textit{The Magma algebra system. I. The user language,}
 J. Symbolic Comput.  {\bf 24}  (1997), 235--265.

\bibitem[Mu]{Mukai} S.\ Mukai, 
 \textit{Finite groups of automorphisms of K3 surfaces and the Mathieu group,}
Invent. Math. {\bf 94} (1988),  183--221. 

\bibitem[MSV]{MSV}F. Malikov, V. Schechtman, and A. Vaintrob, 
\textit{Chiral de {R}ham complex}, Comm. Math. Phys. {\bf 204} (1999),  439--473.

\bibitem[Ni]{Nikulin} Nikulin, V. V.
{\it Finite groups of automorphisms of K\"ahlerian K3 surfaces.} (Russian) 
Trudy Moskov. Mat. Obshch. {\bf 38} (1979), 75--137. 

\bibitem[Ta1]{Ta1} H. Tamanoi, \textit{Elliptic genera and vertex operator super-algebras},
Lecture Notes in Mathematics, 1704, Springer-Verlag, 1999.
		
\bibitem[Ta2]{Ta2} H. Tamanoi, \textit{Elliptic genera and vertex operator super algebras},
   Proc. Japan Acad. Ser. A Math. Sci., {\bf 71} (1995), 177--181.
		
\bibitem[We]{We} H. Weyl, \textit{The Classical Groups: Their Invariants and Representations}, 
Princeton University Press, 1946.

\bibitem[Zh]{Zhou} J. Zhou, \textit{Superconformal Vertex Algebras in Differential Geometry. I}, preprint 2000, arXiv:math/0006201. 

\end{thebibliography}
\end{document}